\documentclass[12pt]{amsart}
\usepackage{amscd,amsmath,amssymb,amsthm,amsfonts,epsfig,graphics}
\usepackage{geometry}
\usepackage{graphicx}
\usepackage{mathrsfs,amssymb}
\usepackage{bm}
\theoremstyle{plain}

\newtheorem{lemma}{Lemma}

\newtheorem{proposition}{Proposition}
\newtheorem{remark}{Remark}

\newtheorem{theorem}{Theorem}
\numberwithin{equation}{section}

\begin{document}
\title[IMPROVED CRITICAL EIGENFUNCTION RESTRICTION ESTIMATES IN 3D]{IMPROVED CRITICAL EIGENFUNCTION RESTRICTION ESTIMATES ON RIEMANNIAN MANIFOLDS WITH CONSTANT NEGATIVE CURVATURE}
\author{Cheng Zhang}
\address{Department of Mathematics\\
Johns Hopkins University\\
Baltimore, MD 21218, USA}
\email{czhang67@math.jhu.edu}
\date{}
\begin{abstract}We show that one can obtain logarithmic improvements of $L^2$ geodesic restriction estimates for eigenfunctions on 3-dimensional compact Riemannian manifolds with constant negative curvature. We obtain a $(\log\lambda)^{-\frac12}$ gain for the $L^2$-restriction bounds, which improves the corresponding bounds of Burq, G\'erard and Tzvetkov \cite{burq}, Hu \cite{hu}, Chen and Sogge \cite{chensogge}. We achieve this by adapting the approaches developed by Chen and Sogge \cite{chensogge}, Blair and Sogge \cite{top}, Xi and the author \cite{xz}. We derive an explicit formula for the wave kernel on 3D hyperbolic space, which improves the kernel estimates from the Hadamard parametrix in Chen and Sogge \cite{chensogge}. We prove detailed oscillatory integral estimates with fold singularities by Phong and Stein \cite{ps} and use the Poincar\'e half-space model to establish bounds for various derivatives of the distance function restricted to geodesic segments on the universal cover $\mathbb{H}^3$.
\end{abstract}

\maketitle

\section{Introduction}
Let $(M,g)$ be a compact $n$-dimensional Riemannian manifold and let $\Delta_g$ be the associated Laplace-Beltrami operator. Let $e_\lambda$ denote the $L^2$-normalized eigenfunction
\[-\Delta_g e_\lambda=\lambda^2e_\lambda,\]
so that $\lambda\ge0$ is the eigenvalue of the operator $\sqrt{-\Delta_g}.$ A classical result on the $L^p$-estimates of the eigenfunctions is due to Sogge \cite{Slp}:
\begin{equation}\label{slp}\|e_\lambda\|_{L^p{(M)}}\le C\lambda^{\delta{(p)}},\end{equation}
where $2\le p\le\infty$ and \[\delta(p)=\begin{cases}
\frac{n-1}{2}(\frac{1}{2}-\frac{1}{p}), &2\le p\le p_c,\\ n(\frac{1}{2}-\frac{1}{p})-\frac{1}{2}, &p_c\le p\le\infty, \end{cases}\]if we set $p_c=\frac{2n+2}{n-1}$.
These estimates \eqref{slp} are saturated on the round sphere $S^n$ by zonal functions for $p\ge p_c$ and for $2<p\le p_c$ by the highest weight spherical harmonics. However, it is expected that \eqref{slp} can be improved for generic Riemannian manifolds. It was known that one can get log improvements for $\|e_\lambda\|_{L^p(M)},\ p_c<p\le \infty,$ when $M$ has nonpositive sectional curvature. Indeed, B\'erard's results \cite{berard} on improved remainder term bounds for the pointwise Weyl law imply that
\[\|e_\lambda\|_{L^\infty(M)}\le C\lambda^\frac{n-1}2({\log\lambda})^{-\frac12}\|e_\lambda\|_{L^2(M)}.\]Recently, Hassell and Tacy \cite{hstc} obtained a similar $(\log\lambda)^{-\frac12}$ gain for all $p>p_c$.

Similar $L^p$-estimates have been established for the restriction of eigenfunctions to geodesic segments.  Let $\it\Pi$ denotes the space of all unit-length geodesics.  The works \cite{burq}, \cite{hu}, \cite{chensogge}(see also \cite{rez} for earlier results on hyperbolic surfaces) showed that
\begin{equation}\label{rest}\sup\limits_{\gamma\in\it\Pi}\Big(\int_\gamma|e_\lambda|^p\,ds\Big)^\frac1p\le C\lambda^{\sigma(n,p)}\|e_\lambda\|_{L^2(M)},
\end{equation}
where \begin{equation}\sigma(2,p)=
\begin{cases}
\frac14, &2\le p\le 4,\\ \frac12-\frac1p, &4\le p\le\infty, \end{cases}
\end{equation}\begin{equation}\sigma(n,p)=\frac{n-1}2-\frac1p,\ {\rm if}\ {p}\ge2\ {\rm and}\ n\ge3.
\end{equation}
It was known that these estimates are saturated by the highest weight spherical harmonics when $n\ge3$ on round sphere $S^n$, as well as in the case of $2\le p\le 4$ when $n=2$, while in this case the zonal functions saturate the bounds for $p\ge 4$.

There are considerable works towards improving \eqref{rest} for the 2-dimensional manifolds with nonpositive curvature. Chen \cite{chen} proved a $(\log\lambda)^{-\frac12}$ gain for all $p>4$. Sogge and Zelditch \cite{sz} and Chen and Sogge \cite{chensogge} showed that one can improve \eqref{rest} for $2\le p\le4$, in the sense that
\begin{equation}
\sup\limits_{\gamma\in\it\Pi}\Big(\int_\gamma|e_\lambda|^p\,ds\Big)^\frac1p=o(\lambda^\frac14).
\end{equation}
Recently, using the Toponogov's comparison theorem, Blair and Sogge \cite{top} obtained log improvements for $p=2$:
\begin{equation}\label{top1}
\sup\limits_{\gamma\in\it\Pi}\Big(\int_\gamma|e_\lambda|^2\,ds\Big)^\frac12\le C\lambda^\frac14(\log\lambda)^{-\frac14}\|e_\lambda\|_{L^2(M)}.
\end{equation}
Inspired by the works \cite{chensogge}, \cite{top}, \cite{loglog}, Xi and the author \cite{xz} was able to deal with the other endpoint $p=4$ and proved a $(\log\log\lambda)^{-\frac18}$ gain for surfaces with nonpositive curvature and a $(\log\lambda)^{-\frac14}$ gain for hyperbolic surfaces
\begin{equation}\label{xz1}
\sup\limits_{\gamma\in\it\Pi}\Big(\int_\gamma|e_\lambda|^4\,ds\Big)^\frac14\le C\lambda^\frac14(\log\lambda)^{-\frac14}\|e_\lambda\|_{L^2(M)}.
\end{equation}
 In the 3-dimensional case, under the assumption of nonpositive curvature, Chen \cite{chen} also proved a $(\log\lambda)^{-\frac12}$ gain for all $p>2$. With the assumption of constant negative curvature, Chen and Sogge \cite{chensogge} showed that \begin{equation}\label{3d}
\sup\limits_{\gamma\in\it\Pi}\Big(\int_\gamma|e_\lambda|^2\,ds\Big)^\frac12 = o(\lambda^{\frac12}).
\end{equation}
Moreover, Hezari and Rivi${\rm \grave{e}}$re \cite{heri} and Hezari \cite{hez} used quantum ergodic methods to get logarithmic improvements at critical exponents in the cases above on negatively curved manifolds for a density one subsequence.

The purpose of this paper is to prove a $(\log\lambda)^{-\frac12}$ gain for the $L^2$ geodesic restriction bounds on 3-dimensional compact Riemannian manifolds with constant negative curvature. We mainly follow the approaches developed in \cite{chensogge}, \cite{top}, \cite{xz}. We derive an explicit formula for the wave kernel on $\mathbb H^3$, which is one of the key steps to get the $(\log\lambda)^{-\frac12}$ gain. We shall lift all the calculations to the universal cover $\mathbb H^3$ and then use the Poincar\'e half-space model to derive the explicit formulas of the mixed derivatives of the distance function restricted to the unit geodesic segments. Then we decompose the domain of the distance function and compute the bounds of various mixed derivatives explicitly, since it was observed in \cite{chensogge} and \cite{xz} that the desired kernel estimates follow from the oscillatory integral estimates and the estimates on the mixed derivatives. Moreover, whether one can get similar logarithmic improvements on 3-dimensional manifolds with nonpositive curvature is still an interesting open problem. One of the technical difficulties is that these manifolds may not have sufficiently many totally geodesic submanifolds (see \cite[p.458]{chensogge}). Throughout this paper, we shall assume that the injectivity radius of $M$ is sufficiently large, and fix $\gamma$ to be a unit length geodesic segment parameterized by arclength.
\begin{theorem}\label{thm1}
Let $(M,g)$ be a 3-dimensional compact Riemannian manifold of constant negative curvature, let $\gamma\subset M$ be a fixed unit-length geodesic segment. Then for $\lambda\gg 1$, there is a constant $C$ such that
\begin{equation}\label{key}\|e_\lambda\|_{L^2(\gamma)}\le C\lambda^\frac12(\log\lambda)^{-\frac12}\|e_\lambda\|_{L^2(M)}.\end{equation}
Moreover, if $\it\Pi$ denotes the set of unit-length geodesics, there exists a uniform constant $C=C(M,g)$ such that
\begin{equation}\label{uniform}\sup\limits_{\gamma\in\it\Pi}\Big(\int_\gamma|e_\lambda|^2\,ds\Big)^\frac12\le C\lambda^\frac12(\log\lambda)^{-\frac12}\|e_\lambda\|_{L^2(M)}.
\end{equation}
\end{theorem}
\begin{remark}\label{rem1}{\rm As a final remark, we must mention a recently posted work of Blair \cite{mb}. He was able to use geometric tools different from ours to establish bounds on the mixed partials of the distance function on the covering manifold restricted to geodesic segements. Then he independently proved \eqref{xz1} for surfaces with general nonpositive curvature and a $(\log\lambda)^{-\frac12+\epsilon}$ gain for \eqref{3d} on 3-dimensional manifolds with constant negative curvature. Moreover, recently Professor C. Sogge pointed out to the author that one may also get a similar $(\log \lambda)^{-\frac12}$ gain for the $L^4$ geodesic restriction estimates on surfaces with strictly negative curvatures by using the G{\"u}nther's comparison theorem and the Hadamard parametrix.}
\end{remark}

\section{Preliminaries}
We start with some standard reductions. Since the uniform bound \eqref{uniform} follows from a standard compactness argument in \cite[p.452]{chensogge}, we only need to prove \eqref{key}. Let $T\gg1$. Let $\rho\in {\mathcal{S}}(\mathbb R)$ such that $\rho(0)=1$ and ${\rm supp}\ \hat\rho\subset[-1/2,1/2]$, then it is clear that the operator $\rho(T(\lambda-\sqrt { - {\Delta _g}}))$ reproduces eigenfunctions, namely
\[\rho(T(\lambda-\sqrt { - {\Delta _g}}))e_\lambda=e_\lambda.\]
Let $\chi=|\rho|^2$. After a standard $TT^*$ argument, we only need to estimate the norm
\begin{equation}\label{cmpmain}
\|\chi(T(\lambda-\sqrt { - {\Delta _g}}))\|_{L^2(\gamma)\rightarrow L^{2}(\gamma)}.
\end{equation}
Choose a bump function $\beta \in C_0^\infty(\mathbb{R})$ satisfying
\[\beta(\tau)=1 \quad {\rm for} \quad |\tau|\le3/2,\quad  {\rm and} \quad \beta(\tau)=0, \quad |\tau|\ge2.\]
By the Fourier inversion formula, we may represent the kernel of the operator $\chi(T(\lambda-\sqrt { - {\Delta _g}}))$ as an operator valued integral
\[\begin{gathered}
  \chi (T(\lambda  - \sqrt { - {\Delta _g}}))(x,y) = \frac{1}
{{2\pi T}}\int {\beta (\tau )\hat \chi (\tau /T){e^{i\lambda \tau }}({e^{ - i\tau \sqrt { - {\Delta _g}}}})(x,y)d\tau }  \hfill \\
   + \frac{1}
{{2\pi T}}\int {(1 - \beta (\tau ))\hat \chi (\tau /T){e^{i\lambda \tau }}({e^{ - i\tau \sqrt { - {\Delta _g}}}})(x,y)d\tau }  = {K_0}(x,y) + {K_1}(x,y). \hfill \\
\end{gathered}
\]
Then one may use a parametrix to estimate the norm of the integral operator associated with the kernel $K_0(\gamma(t),\gamma(s))$ (see \cite[p.455]{chensogge})
\begin{equation}\label{K_0}\|K_0\|_{L^2[0,1]\rightarrow L^{2}[0,1]}\le  C{\lambda}{T^{ - 1}}.\end{equation}
Since the kernel of $\chi (T(\lambda  + \sqrt { - {\Delta _g}}))$ is $O(\lambda^{-N})$ with constants independent of $T$, by Euler's formula we are left to consider the integral operator $S_\lambda$:
\begin{equation}{S_\lambda }h(t) = \frac{1}
{{\pi T}}\int_{ - \infty }^\infty  {\int_{ 0}^{1} {(1 - \beta (\tau ))\hat \chi (\tau /T){e^{i\lambda \tau }}(\cos \tau \sqrt { - {\Delta _g}})(\gamma (t),\gamma (s))h(s)dsd\tau.} }
\end{equation}
As in \cite{chensogge}, \cite{top}, \cite{xz},  we use the Hadamard parametrix and the Cartan-Hadamard theorem to lift the calculations up to the universal cover $(\mathbb{R}^3,\tilde g)$ of $(M,g)$. Let $\Gamma$ denote the group of deck transformations preserving the associated covering map $\kappa :{\mathbb{R}^3} \to M$ coming from the exponential map from $\gamma(0)$ associated with the metric $g$ on $M$. The metric $\tilde g$ is its pullback via $\kappa$. Choose also a Dirchlet fundamental domain, $D \simeq M$, for $M$ centered at the lift $\tilde \gamma(0)$ of $\gamma(0)$. Let $\tilde \gamma(t),\ t\in \mathbb{R}$, satisfy $\kappa(\tilde\gamma(t))=\gamma(t)$, where $\gamma$ is the unit speed geodesic containing the geodesic segment $\{\gamma(t): \ t\in [0,1]\}$. Then $\tilde \gamma(t)$ is also a geodesic parameterized by arclength. We measure the distances in $(\mathbb{R}^3,\tilde g)$ using its Riemannian distance function $d_{\tilde g}(\ \cdot\ ,\ \emph{}\cdot\ )$. Moreover, we recall that if $\tilde x$ denotes the lift of $x\in M$ to $D$, then \[(\cos t\sqrt { - {\Delta _g}} )(x,y) = \sum\limits_{\alpha  \in \Gamma } {(\cos t} \sqrt { - {\Delta _{\tilde g}}} )(\tilde x,\alpha (\tilde y)).\]
Hence for $t\in [0,1]$,
\[{S_\lambda }h(t) = \frac{1}
{{\pi T}}\sum\limits_{\alpha  \in \Gamma } {\int_{\mathbb{R}}  {\int_{ 0}^{1} {(1 - \beta (\tau ))\hat \chi (\tau /T){e^{i\lambda \tau }}(\cos \tau \sqrt { - {\Delta _{\tilde g}}} )(\tilde \gamma (t),\alpha (\tilde \gamma (s)))h(s)\,dsd\tau }.} }\]
As in \cite{top} and \cite{xz}, we denote the $R$-tube about the infinite geodesic $\tilde\gamma$ by
 \begin{equation}\label{tube}{{\rm T}_R}(\tilde \gamma ) = \{ (x,y,z) \in {\mathbb{R}^3}:{d_{\tilde g}}((x,y,z),\tilde \gamma ) \le R\}\end{equation}
and \[{\Gamma _{{{\rm T}_R}(\tilde \gamma )}} = \{ \alpha  \in \Gamma :\alpha (D) \cap {{\rm T}_R}(\tilde \gamma ) \ne \emptyset \}. \]
From now on we fix $R\approx {\rm Inj} M$. We will see that ${{\rm T}_R}(\tilde \gamma )$ plays a key role in the proof of Lemma \ref{lemma1}. Then we decompose the sum
\[{S_\lambda }h(t) = S_\lambda ^{tube}h(t) + {S_\lambda^{osc} }h(t) = \sum\limits_{\alpha  \in {\Gamma _{{{\rm T}_R}(\tilde \gamma )}}}{{S_{\lambda,\alpha}^{tube}h(t)}}  +  \sum\limits_{\alpha  \notin {\Gamma _{{{\rm T}_R}(\tilde \gamma )}}}{{S_{\lambda,\alpha}^{osc}h(t)}}, t\in[0,1]. \]
Then by the finite propagation speed property and $\hat \chi(\tau)=0$ if $|\tau|\ge1$, we have
\[{d_{\tilde g}}(\tilde \gamma (t),\alpha (\tilde \gamma (s)))\le T, s,t\in[0,1].\]
As observed in \cite[p.11]{top}, \begin{equation}\label{tubenum}\# \{ \alpha  \in {\Gamma _{{{\rm T}_R}(\tilde \gamma )}}:{d_{\tilde g}}(0,\alpha (0)) \in [{2^k},{2^{k + 1}}]\}  \le C{2^k}.\end{equation}
Thus the number of nonzero summands in $S_\lambda ^{tube}h(t)$ is $O(T)$ and in $S_\lambda ^{osc}h(t)$ is $O(e^{CT})$.

Given $\alpha \in \Gamma$ set with $s,t\in[0,1]$
\[K_\alpha(t,s) = \frac{1}
{{\pi T}}\int_{ - T}^T {(1 - \beta (\tau ))\hat \chi (\tau /T){e^{i\lambda \tau }}(\cos \tau \sqrt { - {\Delta _{\tilde g}}} )(\tilde \gamma (t),\alpha (\tilde \gamma (s)))d\tau }.\]
When $\alpha=Identity$, one can use the Hadamard parametrix to prove the same bound as \eqref{K_0} (see e.g. \cite{chen}, p. 9)
\begin{equation}\label{KId} \|K_{\rm Id}\|_{L^2[0,1]\rightarrow L^{2}[0,1]}\le  C{\lambda}{T^{ - 1}}.\end{equation}
If $\alpha \ne Identity$, we set $\phi (t,s) = {d_{\tilde g}}(\tilde \gamma (t),\alpha (\tilde \gamma (s))), s,t\in[0,1].$
Then  by finite propagation speed and $\alpha \ne Identity$, we have \begin{equation}2\le \phi(t,s)\le T, \quad {\rm if} \quad s,t\in[0,1].\end{equation}
As in \cite[p.456]{chensogge}, one may use the Hadamard parametrix and stationary phase to show that $|K_\alpha(t,s)|\le C\lambda T^{-1}r^{-1}+e^{CT}$, where $r=d_{\tilde g}(\tilde\gamma(t), \alpha(\tilde\gamma(s)))$. However, we may get a much better estimate for $K_\alpha$. To see this, we need to derive the explicit formula of the wave kernel on hyperbolic space. Without loss of generality, we may assume that $(M,g)$ has constant negative curvature $-1$, which implies that the covering manifold $(\mathbb{R}^3, \tilde g)$ is the hyperbolic space $\mathbb{H}^3$. If we denote the shifted Laplacian operator by \[L= \Delta _{\tilde g}+\frac{(n-1)^2}{4}=\Delta _{\tilde g}+1\quad \quad ({\rm for}\ n=3),\]which has the property Spec$(-L)=[0, \infty)$, then there are exact formulas for various functions of $L$ (see e.g. \cite[Chapter 8, (5.15)]{taylor}). Indeed,
\[h(\sqrt{-L})\delta_y(x)=-\frac1{(2\pi)^{3/2}}\frac1{\sinh r}\frac{\partial}{\partial r}\hat{h}(r),\]where $\hat{h}$ is the Fourier transform defined by
\[\hat{h}(r)=\frac1{\sqrt{2\pi}}\int_{-\infty}^{\infty}h(k)e^{-irk}dk.\]
If $h(k)=\frac{\sin(tk)}{k}$, then $\hat{h}(r)=\frac{\sqrt{2\pi}}{2}\textbf{1}_{\{r\le |t|\}}.$ Hence, for $t>0$,
\begin{equation}\label{shiftL}\frac{\sin t\sqrt{-L}}{\sqrt{-L}}\delta_y(x)=\frac{\delta(t-r)}{4\pi\sinh r},\end{equation}
where $x,\ y\in \mathbb{H}^3$ and $r=d_{\tilde g}(x,y)$. Differentiating it yields
\begin{equation}\label{cosL}\cos t\sqrt{-L}\ \delta_y(x)=\frac{\delta'(t-r)}{4\pi\sinh r}.\end{equation} Recall the following relation between $L$ and $\Delta _{\tilde g}$ (see e.g. \cite[Proposition 2.1]{mt})
\begin{equation}\label{relation}\cos t\sqrt{-\Delta _{\tilde g}}=\cos t\sqrt{-L}-t\int_0^t{\frac{J_1(\sqrt{t^2-s^2})}{\sqrt{t^2-s^2}}\cos s\sqrt{-L}ds},\end{equation}
where $J_1(v)$ is the Bessel function
\[J_1(v)=\sum_{k=0}^{\infty}{\frac{(-1)^k}{k!(k+1)!}\Big(\frac{v}{2}\Big)^{2k+1}}.\]
We plug \eqref{cosL} into the relation \eqref{relation} to see that for $t>0$,
\[\cos t\sqrt{-\Delta _{\tilde g}}\ \delta_y(x)=\frac{\delta'(t-r)}{4\pi\sinh r}-t\int_0^t{\frac{J_1(\sqrt{t^2-s^2})}{\sqrt{t^2-s^2}}\frac{\partial_s\delta(s-r)}{4\pi\sinh r}ds},\]
Thus, integrating by parts and noting that $\cos t\sqrt{-\Delta _{\tilde g}}$ is even in $t$, we get the following explicit formula for the wave kernel ``$\cos t\sqrt{ - {\Delta _{\tilde g}}}(x,y)$'' on $\mathbb{H}^3$
\begin{equation}\label{wavekernel}\cos t\sqrt{ - {\Delta _{\tilde g}}}\ \delta_y(x)=\frac1{4\pi\sinh r}\Big[\delta'(|t|-r)-J'_1(0)|t|\delta(|t|-r)-\frac{r|t|G'(\sqrt{t^2-r^2})}{\sqrt{t^2-r^2}}{\textbf{1}}_{\{r\le |t|\}}\Big],\end{equation}
where $t\in {\mathbb{R}}\setminus \{0\}$, and $G(v)=J_1(v)/v$ is an entire function of $v^2$, satisfying
\begin{equation}\label{Gv}G(v)\sim Cv^{-3/2}\cos\Big(v-\frac{3\pi}{4}\Big)+\cdot\cdot\cdot, \ {\rm as}\ v\to +\infty.\end{equation}
\begin{lemma}\label{lemmaK}If $\alpha\ne Identity$, we have
\[|K_\alpha(t,s)|\le C\lambda T^{-1}e^{-r/2},\ {\rm for}\ t,s\in[0,1],\]
where $r=d_{\tilde g}(\tilde\gamma(t), \alpha(\tilde\gamma(s)))\ge1$ and $C$ is a constant independent of $T$ and $r$.
\end{lemma}
Using this lemma and \eqref{tubenum}, we get
\begin{equation}\Big|\sum\limits_{\alpha  \in {\Gamma _{{{\rm T}_R}(\tilde \gamma )}}\setminus \{\rm Id\}}{{K_{\alpha}(t,s)}}  \Big| \le C\lambda T^{-1}\sum\limits_{1\le 2^k\le T} {2^k e^{-2^k/2}}\le C\lambda T^{-1}.\end{equation}
Consequently, by Young's inequality and the estimate on $K_{\rm Id}$ \eqref{KId} we have
\begin{equation}\label{stube}\|S_\lambda^{tube}\|_{L^2[0,1]\rightarrow L^{2}[0,1]}\le C{\lambda}T^{ - 1}.\end{equation}

\begin{proof}[Proof of Lemma \ref{lemmaK}]
Since the formula of the wave kernel \eqref{wavekernel} consists of 3 terms, we should estimate their contributions separately. Integrating by parts yields
\begin{equation}\begin{aligned}\Big|\int_{-T}^{T}{(1-\beta(\tau))\hat{\chi}(\tau/T)e^{i\lambda\tau}\delta'(|\tau|-r)d\tau}\Big|&\le \sum_{\tau=\pm r} \Big|\frac{d}{d\tau}\Big[(1-\beta(\tau))\hat{\chi}(\tau/T)e^{i\lambda\tau}\Big]\Big|\\
&\le C\lambda,\end{aligned}\end{equation}
since $\beta,\ \hat{\chi}\in \mathcal{S}(\mathbb{R})$.
Similarly,
\begin{equation}\begin{aligned}\Big|\int_{-T}^{T}{(1-\beta(\tau))\hat{\chi}(\tau/T)e^{i\lambda\tau}|\tau|\delta(|\tau|-r)d\tau}\Big|&=\Big|\sum_{\tau=\pm r}(1-\beta(\tau))\hat{\chi}(\tau/T)e^{i\lambda\tau}|\tau|\Big|\\
&\le Cr.\end{aligned}\end{equation}
Noting that $J_1(v)$, $J'_1(v)$ are uniformly bounded for $v\in \mathbb{R}$ and $G(v)$ is an entire function of $v^2$, we see that $G'(v)/v$ is also uniformly bounded for $v\in \mathbb{R}$. Moreover, by \eqref{Gv}, there is some $N\gg1$ such that \[ |G'(v)/v|\le Cv^{-5/2},\quad {\rm for}\ v>N.\] This gives
\begin{equation}\begin{aligned}&\Big|\int_{-T}^{T}{(1-\beta(\tau))\hat{\chi}(\tau/T)e^{i\lambda\tau}\frac{r|\tau|G'(\sqrt{\tau^2-r^2})}{\sqrt{\tau^2-r^2}}{\textbf{1}}_{\{r\le |\tau|\}}d\tau}\Big|\\
&\le Cr\int_{|\tau|\ge r}|\tau|\left|\frac{G'(\sqrt{\tau^2-r^2})}{\sqrt{\tau^2-r^2}}\right|d\tau\\
&\le Cr\Big(\int_0^{N}|\rho+r|d\rho+\int_{N}^{\infty}{|\rho+r||\rho|^{-5/2}d\rho}\Big)\\
&\le Cr(C+Cr),\end{aligned}\end{equation}
where $\rho=|\tau|-r$.
Hence
\[|K_\alpha(t,s)|\le \frac{C\lambda+Cr+Cr^2}{T\sinh r}\le C\lambda T^{-1}e^{-r/2}.\]
\end{proof}
\begin{remark} {\rm As pointed out in Remark \ref{rem1}, one may also obtain Lemma \ref{lemmaK} by using the Hadamard parametrix and  the G{\"u}nther's comparison theorem.}\end{remark}

\section{Proof of the main theorem}
Now we are left to estimate the kernels $K_{\alpha}(t,s)$ with $\alpha\notin {\Gamma _{{{\rm T}_R}(\tilde \gamma )}} $. From now on, we assume that $\alpha\notin {\Gamma _{{{\rm T}_R}(\tilde \gamma )}} $. First of all, we need a slight variation of the oscillatory integral theorem in \cite[Proposition 2]{xz}. Indeed, it is a detailed version of the estimates by Phong and Stein \cite{ps} on the oscillatory integrals with fold singularities.
\begin{proposition}\label{oscint}Let $a\in C_0^\infty(\mathbb{R}^2)$, let $\phi\in C^\infty(\mathbb{R}^2)$ be real valued and $\lambda>0$, set
\[{T_\lambda }f(t) = \int_{ - \infty }^\infty  {{e^{i\lambda \phi (t,s)}}a} (t,s)f(s)ds,\ f\in C_0^\infty(\mathbb{R}).\]
 If $\phi_{st}^{''}\ne 0$ on ${\rm supp}$ a, then \[{\left\| {{T_\lambda }f} \right\|_{{L^2}(\mathbb{R})}} \le {C_{a,\phi }}\lambda^{-\frac12}{\left\| f \right\|_{{L^2}(\mathbb{R})}},\]
where \begin{equation}\label{const1}C_{a,\phi}=C{\rm diam}({\rm supp}\ a)^{\frac12}\Bigg\{{\|a\|_\infty+\frac{ {\sum\limits_{0 \le i,j \le 2} {\|\partial _t^ia\|_\infty\|\partial _t^j\phi _{st}^{''}\|_\infty} }}
{{{{\inf |\phi _{st}^{''}|^2}}}}}\Bigg\}.\end{equation}
Assume  ${\rm supp}$ a is contained in some compact set $F\subseteq {\mathbb R}^2$. Denote the ranges of $t$ and $s$ in $F$ by $F_t\subseteq {\mathbb R}$ and $F_s\subseteq {\mathbb R}$ respectively. If for any $s\in F_s$, there is a unique $t_c=t_c(s)\in F_t$ such that $\phi_{st}^{''}(t_c,s)=0$, and if $\phi_{stt}^{'''}(t_c,s)\ne0$ on $F_s$, then
\[{\left\| {{T_\lambda }f} \right\|_{{L^2}(\mathbb{R})}} \le {C_{a,\phi }'}\lambda^{-\frac14}{\left\| f \right\|_{{L^2}(\mathbb{R})}},\]
where
\begin{equation}\label{const2}C_{a,\phi}'=C{\rm diam}({\rm supp}\ a)^{\frac14}\Bigg\{\|a\|_\infty+\frac{\sum\limits_{0 \le i,j \le 2} {\|\partial _t^ia\|_\infty \|\partial _t^j\phi _{st}^{''}\|_\infty}}
{{\inf {|\phi _{st}^{''}/(t - {t_c(s)})|^2}}}\Bigg\}.
\end{equation}
Dually, if for any $t\in F_t$, there is a unique $s_c=s_c(t)\in F_s$ such that $\phi_{st}^{''}(t,s_c)=0$, and if $\phi_{tss}^{'''}(t,s_c)\ne0$ on $F_t$, then
\[{\left\| {{T_\lambda }f} \right\|_{{L^2}(\mathbb{R})}} \le {C_{a,\phi }''}\lambda^{-\frac14}{\left\| f \right\|_{{L^2}(\mathbb{R})}},\]
where
\begin{equation}\label{const3}C_{a,\phi}''=C{\rm diam}({\rm supp}\ a)^{\frac14}\Bigg\{\|a\|_\infty+\frac{\sum\limits_{0 \le i,j \le 2} {\|\partial _s^ia\|_\infty \|\partial _s^j\phi _{st}^{''}\|_\infty}}
{{\inf {|\phi _{st}^{''}/(s - {s_c(t)})|^2}}}\Bigg\}.
\end{equation}
The $L^\infty$-norm and the infimum are taken on ${\rm supp}$ a. The constant $C>0$ is independent of $\lambda$, $a$, $\phi$ and $F$.
\end{proposition}
\begin{proof}
Noting that the first part is due to non-stationary phase (see \cite[p.15]{xz}) and the third part simply follows from duality, we only need to prove the second part. As in \cite[p.15]{xz}, by a $TT^*$ argument, it suffices to estimate the kernel of $T_\lambda^* T_\lambda$
\[K(s,s')=\int {{e^{i\lambda (\phi (t,s) - \phi (t,s'))}}a(t,s)\overline {a(t,s')} dt}. \]
Let \[\varphi (t,s,s') = \frac{{\phi (t,s) - \phi (t,s')}}
{{s - s'}},\ {\rm for}\ s\ne s',\ {\rm and}\ \varphi (t,s,s)=\phi'_s(t,s),\]\[\tilde a(t,s,s') = a(t,s)\overline {a(t,s')}. \]
Then the kernel has the form
\begin{equation}\label{kernelK}K(s,s')=\int {{e^{i\lambda (s-s')\varphi(t,s,s')}}\tilde a(t,s,s')} dt. \end{equation}
Using the mean value theorem, we have $\varphi'_t(t,s,s')=\phi''_{st}(t,s'')$, where $s''$ is a number between $s$ and $s'$. By our assumptions, we see that there is a unique point $t_c(s'')\in F_t$ such that $\phi''_{st}(t_c(s''),s'')=0$, and $\phi_{stt}^{'''}(t_c(s''),s'')\ne0$. Let $\theta >0$. Select $\eta \in C_0^\infty(\mathbb{R})$ satisfying $\eta(t)=1, |t|\le 1$, and $\eta (t)=0, |t|\ge 2.$ Then we decompose the oscillatory integral into two parts. First, \[\Big|\int {{e^{i\lambda (s - s')\varphi }}} \tilde a\eta ((t-t_c(s''))/\theta )dt\Big| \le 4\theta \| a \|_\infty ^2.\]
Then integrating by parts yields if $s\neq s'$,
\[\begin{gathered}
  \Big|\int {{e^{i\lambda (s - s')\varphi }}\tilde a(1 - \eta ((t - {t_c(s'')})/\theta ))dt} \Big| \hfill \\
   \le {(\lambda |s - s'|)^{ - 2}}\int\limits_{|t - {t_c(s'')}| > \theta } {\left| {\frac{\partial }
{{\partial t}}\left( {\frac{1}
{{\varphi'_t}}\frac{\partial }
{{\partial t}}\left( {\frac{{\tilde a(1 - \eta ((t - {t_c(s'')})/\theta ))}}
{{\varphi'_t}}} \right)} \right)} \right|dt}  \hfill \\
   \le\frac{C{{{\Big( {\sum\limits_{0 \le i,j \le 2} {||\partial _t^ia|{|_\infty }||\partial _t^i\phi _{st}^{''}|{|_\infty }} } \Big)}^2}}}
{{(\lambda |s - s'|)^{2}}\cdot{\inf {{(|\phi _{st}^{''}|/|t - {t_c(s)}|)}^4}}}\int\limits_{|t - t_c(s'')| > \theta } {(|t - {t_c(s'')}{|^{ - 4}} + {\theta ^{ - 2}}|t - {t_c(s'')}{|^{ - 2}})dt}  \hfill \\
   \le C{\theta ^{ - 3}}{(\lambda |s - s'|)^{ - 2}}\frac{{{{\Big( {\sum\limits_{0 \le i,j \le 2} {||\partial _t^ia|{|_\infty }||\partial _t^i\phi _{st}^{''}|{|_\infty }} } \Big)}^2}}}
{{\inf {{(|\phi _{st}^{''}|/|t - {t_c(s)}|)}^4}}}, \hfill \\
\end{gathered}\]
where $C$ is a constant independent of $\lambda$, $a$, $\phi$ and $F$.
If we set $\theta=(\lambda|s-s'|)^{-\frac12}$, then
\[|K(s,s')|\le C\Bigg\{\|a\|_\infty^2+\frac{{{{\Big( {\sum\limits_{0 \le i,j \le 2} {||\partial _t^i a|{|_\infty }||\partial _t^j\phi _{st}^{''}|{|_\infty }} } \Big)}^2}}}
{{\inf (|\phi _{st}^{''}|/|t-t_c(s)|)^4}}\Bigg\}(\lambda |s-s'|)^{-\frac12},\ \ {\rm if}\ s\neq s'.\]
Hence, \[\int {|K(s,s')|ds}\le C_{a,\phi}^{'2}\lambda^{-\frac12},\]
which completes the proof by Young's inequality.
\end{proof}
From now on, we will use $C$ to denote various positive constants independent of $T$. Using the Hadamard parametrix and stationary phase \cite[p.446]{chensogge}, we can write \[K_\alpha(t,s) = w(\tilde \gamma (t),\alpha (\tilde \gamma (s)))\sum\limits_ \pm  {{a_ \pm }(T,\lambda ;\phi (t,s)){e^{ \pm i\lambda \phi (t,s)}} + R_\alpha(t,s)}, \]
where $|w(x,y)|\le C$, and for each $j=0,1,2,...$, there is a constant $C_j$ independent of $T$, $\lambda \ge 1$ so that \begin{equation}\label{amplitude}\left| {\partial _r^j{a_ \pm }(T,\lambda ;r)} \right| \le {C_j}{T^{ - 1}}{\lambda}{r^{ - 1 - j}}, \ r\ge1.\end{equation}
From the Hadamard parametix with an estimate on the remainder term (see \cite{hangzhou}), we see that with a uniform constant $C$ \[|R_\alpha(t,s)|\le e^{CT}.\]
Noting that ${\rm diam}({\rm supp}\ a_\pm)\le2$ and we have good control on the size of $a_\pm$ and its derivatives by \eqref{amplitude}, it remains to estimate the size of $\phi_{st}^{''}$ and its derivatives. Without loss of generality, we may assume that $(M,g)$ is a compact 3-dimensional Riemannian manifold with constant curvature equal to $-1$. As in \cite{xz}, we will compute the various mixed derivatives of the distance function explicitly on its universal cover $\mathbb{H}^3$. We consider the Poincar\'e half-space model
\[\mathbb{H}^3=\{(x,y,z)\in \mathbb{R}^3:z>0\},\]
with the metric  $ds^2=z^{-2}(dx^2+dy^2+dz^2)$.
Recall that the distance function for the Poincar\'e half-space model is given by
\[{\rm dist}((x_1,y_1,z_1),(x_2,y_2,z_2))={\rm arcosh}\left( {1{\text{ + }}\frac{(x_2- x_1)^2 + (y_2 -y_1)^2+(z_2-z_1)^2}
{{2{z_1}{z_2}}}} \right),\]
where arcosh is the inverse hyperbolic cosine function\[{\rm arcosh}(x)={\rm ln}(x+\sqrt{x^2-1}),\,x\ge1.\]
Moreover, the geodesics are the straight vertical rays normal to the $z=0$-plane and the half-circles normal to the $z=0$-plane with origins on the $z=0$-plane. Without loss of generality, we may assume that $\tilde\gamma$ is the $z$-axis. Let $\tilde\gamma(t)=(0,0,e^t),\ t\in \mathbb{R},$ be the infinite geodesic parameterized by arclength. Our unit geodesic segment is given by $\tilde\gamma(t),\ t\in[0,1].$ Then its image $\alpha(\tilde\gamma(s)),\ s\in [0,1],$ is a unit geodesic segment of $\alpha(\tilde\gamma)$. As before, we denote the distance function ${d_{\tilde g}}(\tilde \gamma (t),\alpha (\tilde \gamma (s)))$ by $\phi(t,s)$. Since we are assuming  $\alpha\notin {\Gamma _{{{\rm T}_R}(\tilde \gamma )}} $, we have \begin{equation}\label{phileT}2\le \phi(t,s)\le T, \quad {\rm if} \quad s,\ t\in[0,1].\end{equation}
If $\tilde\gamma$ and $\alpha(\tilde\gamma)$ are contained in a common plane, it is reduced to the 2-dimensional case. We recall the following lemma from \cite[Lemma 5, 6]{xz}, where $\tilde\gamma(t)=(0,e^t)$ in the Poincar\`e half-plane model.
\begin{lemma}\label{lemma0}Let $\alpha  \notin {\Gamma _{{{\rm T}_R}(\tilde \gamma )}}$. If $\alpha(\tilde\gamma)\cap\tilde\gamma=\emptyset$, we have
\[{\rm inf}\ |\phi_{st}^{''}|\ge e^{-CT}.\]
Assume that $\alpha(\tilde\gamma)$ is a half-circle intersecting $\tilde\gamma$ at the point $(0,e^{t_0}),\ t_0\in \mathbb{R}$. If $t_0\notin[-1,2]$, which means the intersection point $(0,e^{t_0})$ is outside some neighbourhood of the geodesic segment $\{\tilde \gamma(t):\,t\in[0,1]\}$, then we also have
\[{\rm inf}\ |\phi_{st}^{''}|\ge e^{-CT}.\]
If $t_0\in[-1,2]$, then
\[{\rm inf}\ |\phi_{st}^{''}/(t-t_0)|\ge e^{-CT}.\]
Moreover,
\[\|\phi_{st}^{''}\|_\infty+\|\phi_{stt}^{'''}\|_\infty+\|\phi_{sttt}^{''''}\|_\infty\le e^{CT},\]
where  $C>0$ is independent of $T$. The infimum and the norm are taken on the unit square $\{(t,s)\in \mathbb{R}^2:t,s\in[0,1]\}$.
\end{lemma}
From now on, we assume that $\alpha\notin {\Gamma _{{{\rm T}_R}(\tilde \gamma )}} $, and $\tilde\gamma$ and $\alpha(\tilde\gamma)$ are not contained in a common plane. Without loss of generality, we set $a\ge 0$, $r>0$, and $\beta\in(0,\frac{\pi}{2}]$. Indeed, one can properly choose a coordinate system to achieve this. Let $\gamma_1(t)=(0,0,e^t)$, and $\gamma_2(s)=(a+\frac{1-e^{2s}}{1+e^{2s}}r{\rm cos}\beta, \frac{1-e^{2s}}{1+e^{2s}}r{\rm sin}\beta,\frac{2re^s}{1+e^{2s}})$. It is not difficult to verify that both of them are parameterized by arclength. Assume that
 \[\{\tilde\gamma(t): t\in[0,1]\}=\{\gamma_1(t): t\in[0,1]\},\quad \{\alpha(\tilde\gamma(s)): s\in[0,1]\}=\{\gamma_2(s): s\in I\},\]where $I$ is some unit closed interval of $\mathbb{R}$.  Here $\gamma_2(s),\ s\in{\mathbb R}$, is a half circle centered at $(a,0,0)$ with radius $r$. $\beta$ is the angle between the y-axis and the normal vector of the plane containing the half circle. Moreover, these two geodesics are contained in a common plane when $\beta=0$. See Figure \ref{1}.  \begin{figure}
  \centering
    \includegraphics[height=8cm]{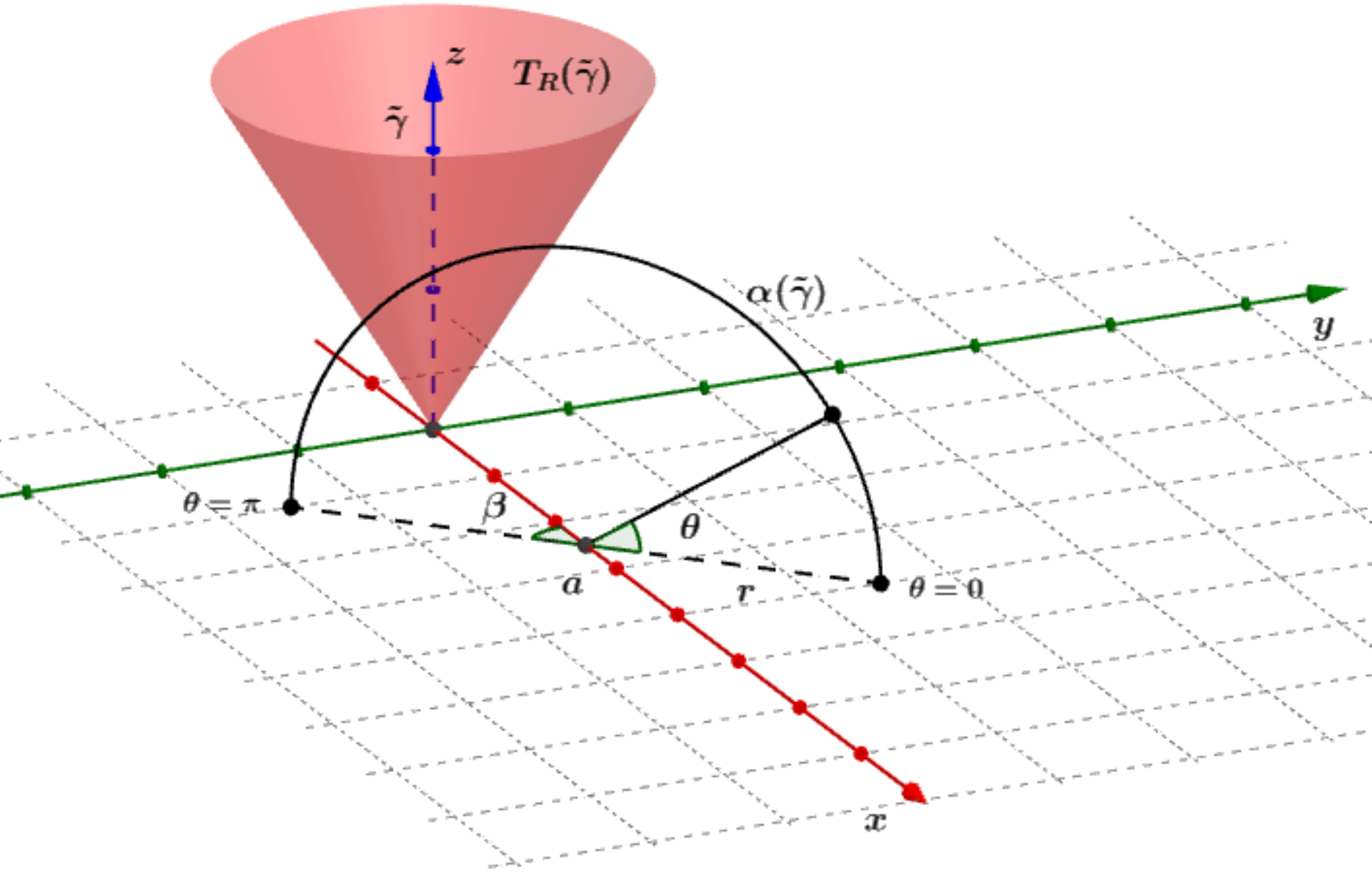}
\caption{Poincar\'e half-space model}
  \label{1}
\end{figure}

Now we are ready to compute $\phi_{st}^{''}$ explicitly and analyze its zero set. For simplification, we denote \[d_1=\sqrt{a^2+r^2-2ar{\rm cos}\beta}\quad {\rm and}\quad  d_2=\sqrt{a^2+r^2+2ar{\rm cos}\beta}.\] Direct computation gives
\[\phi(t,s)={d_{\tilde g}}(\gamma_1 (t), \gamma_2 (s))={\rm arccosh}\Big(\frac{A}{4re^{s+t}}\Big),\ t\in[0,1],\ s\in I,\]
where $A=e^{2s+2t}+e^{2t}+d_1^2e^{2s}+d_2^2$.
Taking derivatives yields
\begin{equation}\label{phist}\phi_{st}^{''}=\frac{16re^{2s+2t}[(a{\rm cos}\beta-r)(e^{2s+2t}+d_2^2)+(a{\rm cos}\beta+r)(e^{2t}+d_1^2e^{2s})]}{(A^2-16r^2e^{2s+2t})^{3/2}}.\end{equation}
The computation is technical. To see \eqref{phist}, we write \[e^{s+t}{\rm cosh}\phi=\frac{A}{4r}.\]Taking derivatives on both sides, we obtain
\begin{equation}\label{eqphist}(\phi'_t+\phi'_s+\phi_{ts}^{''}){\rm sinh}\phi+(1+\phi'_t\phi'_s){\rm cosh \phi}=e^{s+t}/r.\end{equation}
Denote $P=e^{s+t}$, $Q=d_1^2e^{s-t}$, $R=e^{t-s}$, and $S=d_2^2e^{-s-t}$. Since \[4r{\rm cosh}\phi=P+Q+R+S,\]
 taking derivatives yields
\[4r\phi'_t{\rm sinh}\phi=P-Q+R-S,\quad 4r\phi'_s{\rm sinh}\phi=P+Q-R-S.\]
Then we multiply both sides of \eqref{eqphist} by $4r^2({\rm sinh}\phi)^2$ and use the hyperbolic trigonometric identity $({\rm sinh}\phi)^2=({\rm cosh}\phi)^2-1$ to obtain
\[4r^2({\rm sinh}\phi)^3\phi_{st}^{''}=(a\cos\beta-r)(P+S)+(a\cos\beta+r)(Q+R).\]
This gives our desired expression \eqref{phist}.
\begin{figure}
  \centering
    \includegraphics[width=15cm]{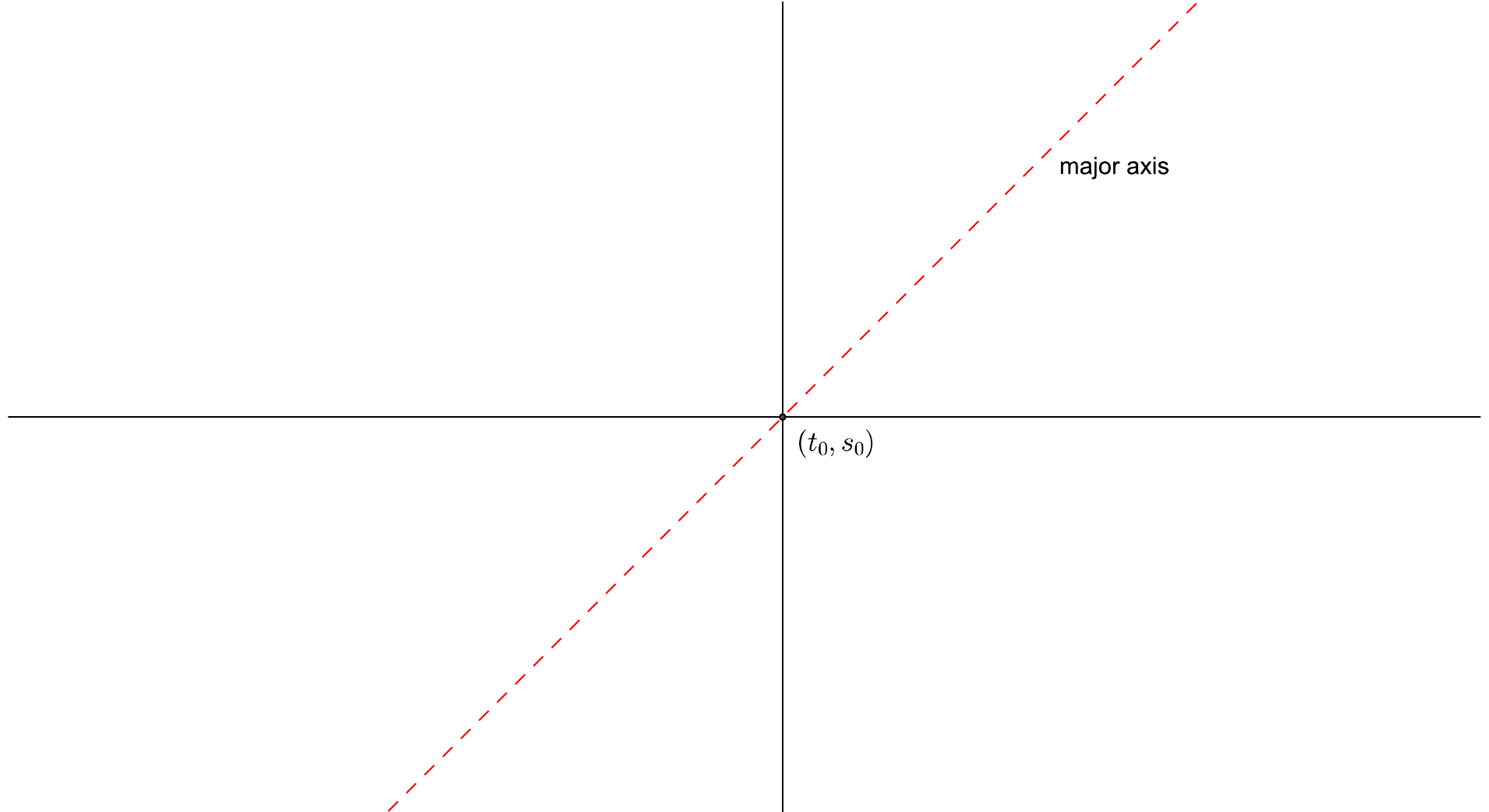}
\caption{Zero set of $\phi_{st}^{''}$, $\beta=\frac{\pi}{2}$}
  \label{2}
  \centering
    \includegraphics[width=15cm]{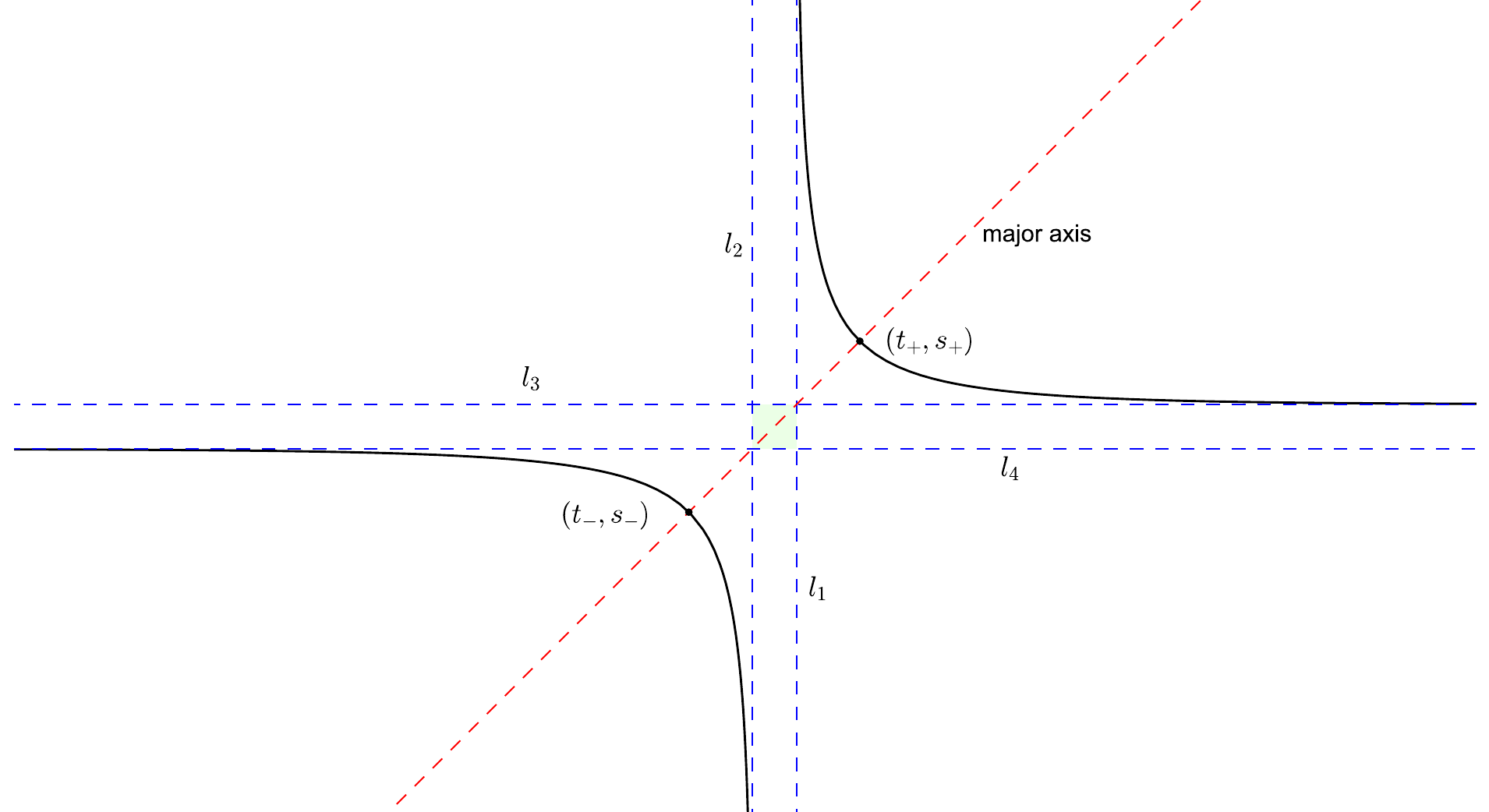}
\caption{Zero set of $\phi_{st}^{''}$, $\beta\in(0,\frac{\pi}{2})$}
  \label{3}
\end{figure}

We denote the zero set of $\phi_{st}^{''}$ by $Z$. Clearly, if $r\le a{\rm cos}\beta$, then $Z=\emptyset$. Assume that $r> a{\rm cos}\beta$. In the interesting special case $\beta=\frac{\pi}{2}$,   \[Z=\{(t,s)\in \mathbb{R}^2:t=t_0\  {\rm or}\ s=s_0\},\]
where $e^{2t_0}=a^2+r^2$ and $e^{2s_0}=1$. See Figure \ref{2}. In this case, we can easily see that $\phi_{stt}^{'''}$ and $\phi_{tss}^{'''}$ vanish at the point $(t_0,s_0)$, as observed in \cite[p.454]{chensogge}. In general, if $0<\beta\le \frac{\pi}{2}$, we have
\begin{equation}\label{zeroset}Z=\{(t,s)\in \mathbb{R}^2:(e^{2t}-X_0)(e^{2s}-Y_0)=B\},\end{equation}
where \begin{equation}\label{defXYB}Y_0=\frac{r+a{\rm cos}\beta}{r-a{\rm cos}\beta},\  X_0=d_1^2Y_0,\ B=\frac{4a^3r{\rm cos}\beta{\rm sin}^2\beta}{(r-a{\rm cos}\beta)^2},\end{equation}and \begin{equation}\label{XYB}X_0Y_0-B=d_2^2.\end{equation}
When $\beta\in (0,\frac\pi2)$, the set $Z$ consists of two disconnected curves. See Figure \ref{3}. It has four different asymptotes:
 \[l_1:\ t=\ln \sqrt{X_0},\quad l_2:\ t=\ln \sqrt{X_0-B/Y_0},\]
 \[l_3:\ s=\ln \sqrt{Y_0},\quad l_4:\ s=\ln \sqrt{Y_0-B/X_0}.\]
 They intersect at four points, which constitute the ``central square'' in Figure \ref{3}. Clearly, the ``central square'' converges to the point $(t_0,s_0)$ as $\beta \to \frac{\pi}{2}$. We set
\begin{equation}\label{tspn}e^{2t_\pm}=X_0\pm\sqrt{\frac{BX_0}{Y_0}} \ \ {\rm and} \ \ e^{2s_\pm}=Y_0\pm\sqrt{\frac{BY_0}{X_0}}.\end{equation}\begin{figure}
  \centering
    \includegraphics[width=15cm]{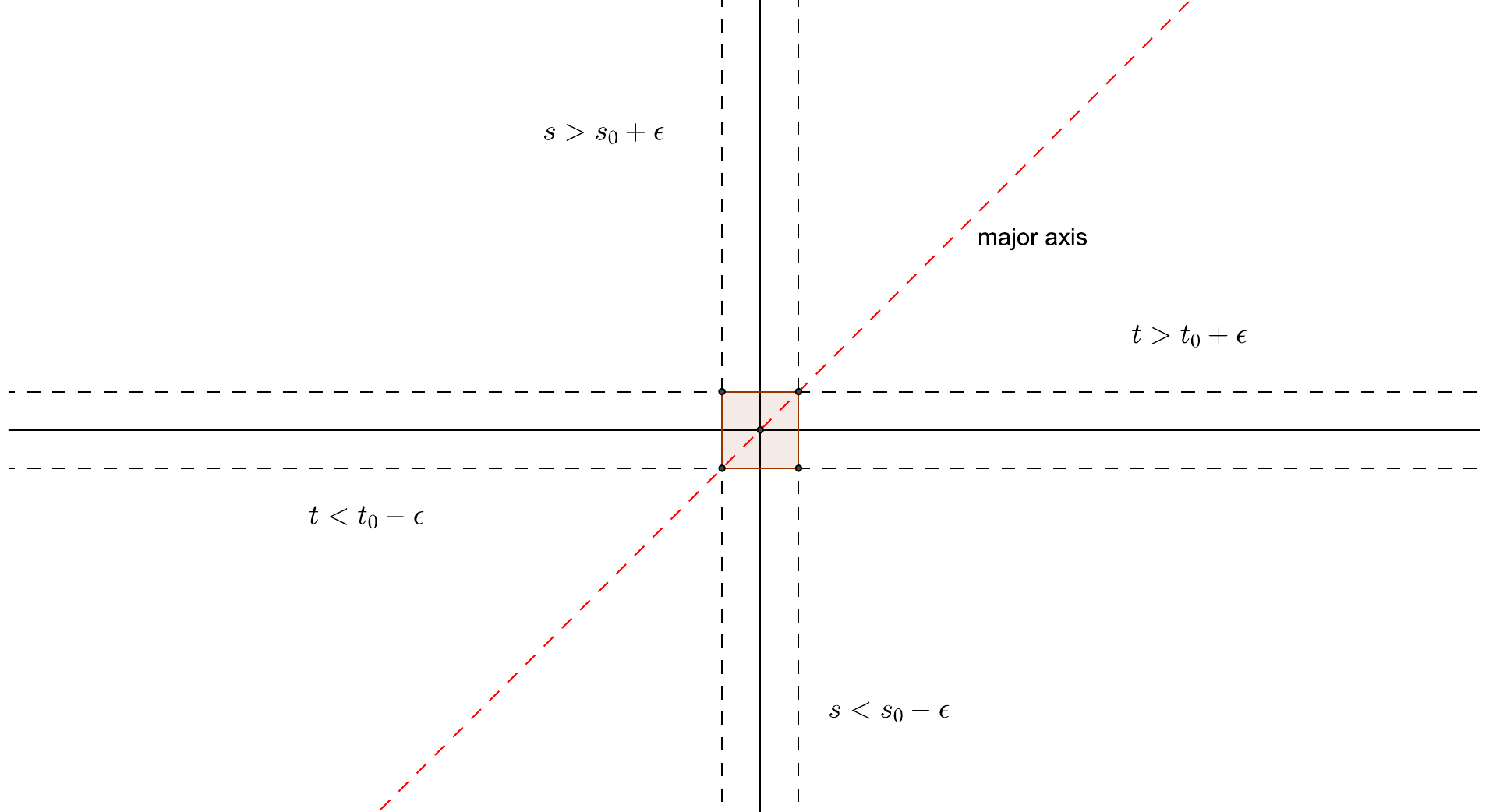}
\caption{$Z_\epsilon$ and its decomposition, $\beta=\frac{\pi}{2}$}
  \label{nbhd1}
  \centering
    \includegraphics[width=15cm]{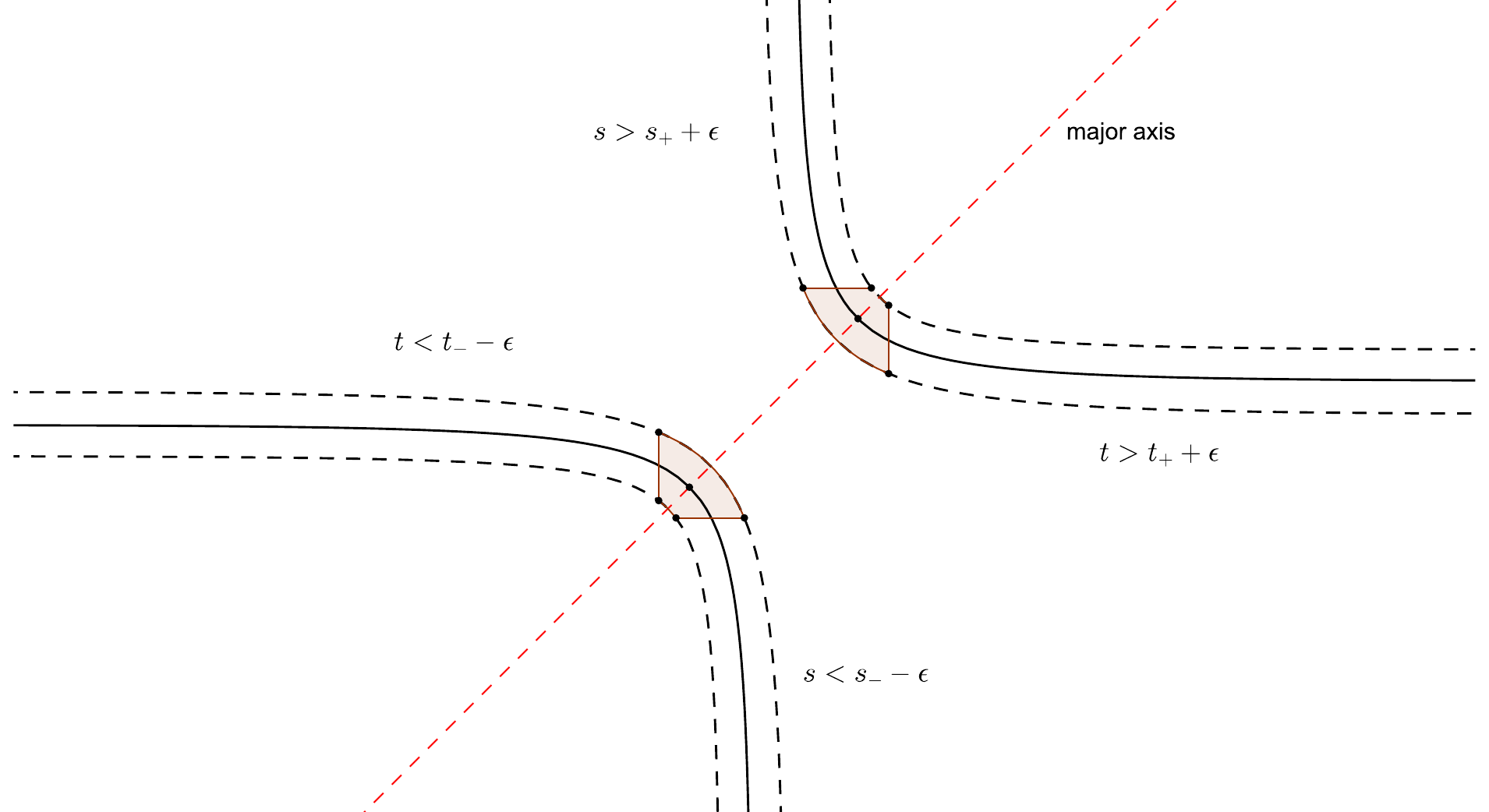}
\caption{$Z_\epsilon$ and its decomposition, $\beta\in(0,\frac{\pi}{2})$}
  \label{nbhd2}
\end{figure}
The points $(t_+,s_+)$ and $(t_-,s_-)$ are a pair of vertices of $Z$ in Figure \ref{3}. They both converge to $(t_0, s_0)$ as $\beta \to \frac{\pi}{2}$. A simple computation shows that the straight line passing through these two vertices, namely the ``major axis'', is parallel to the straight line $t-s=0$. This fact makes the ``restriction trick'' work in the proof of Lemma \ref{lemma1}. Moreover, if $s>s_+$ or $s<s_-$, there is a unique $t_c=t_c(s)$ such that $(t_c,s)\in Z$. If $t>t_+$ or $t<t_-$, there is a unique $s_c=s_c(t)$ such that $(t,s_c)\in Z$. These two facts are related to the oscillatory integral estimates in Proposition \ref{oscint}. Indeed, one can see from \eqref{zeroset} that
\begin{equation}\label{tcsc}e^{2t_c(s)}= X_0+\frac{B}{e^{2s}-Y_0},\quad e^{2s_c(t)}=Y_0+\frac{B}{e^{2t}-X_0}.\end{equation}
Given $0<\epsilon\ll 1$, we denote the  $\epsilon$-neighbourhood  of $Z$ by \[Z_\epsilon=\{(t,s)\in \mathbb{R}^2: {\rm dist}((t,s),Z)\le \epsilon\}.\] In particular, we set $Z_\epsilon=\emptyset$ if $Z=\emptyset$. See Figure \ref{nbhd1} and \ref{nbhd2}. We decompose the domain $[0,1]^2$ of the phase function into 4 parts:

(1) Non-stationary phase part: $[0,1]^2 \setminus Z_\epsilon$;

(2) Left folds part: $[0,1]^2\cap \{(t,s)\in Z_\epsilon: s>s_++\epsilon \ {\rm or} \ s<s_--\epsilon\}$;

(3) Right folds part: $[0,1]^2\cap \{(t,s)\in Z_\epsilon: t>t_++\epsilon \ {\rm or} \ t<t_--\epsilon\}$;

(4) Young's inequality part: $[0,1]^2\cap Z_\epsilon \cap ([t_--\epsilon, t_++\epsilon]\times[s_--\epsilon, s_++\epsilon])$.

\begin{lemma}\label{lemma1}Let $\alpha  \notin {\Gamma _{{{\rm T}_R}(\tilde \gamma )}}$. Assume that $\tilde \gamma$ and $\alpha(\tilde \gamma)$ are not contained in a common plane. Then we have
\[{\rm inf}\ |\phi_{st}^{''}|\ge \epsilon^2e^{-CT},\]
where the infimum is taken on $[0,1]^2 \setminus Z_\epsilon$.
If $Z\ne \emptyset$,  then we have \[{\rm inf}\ |\phi_{st}^{''}|/|t-t_c(s)|\ge \epsilon e^{-CT},\]
where the infimum is taken on $[0,1]^2\cap \{(t,s)\in Z_\epsilon: s>s_++\epsilon \ {\rm or} \ s<s_--\epsilon\}$,
and \[{\rm inf}\ |\phi_{st}^{''}|/|s-s_c(t)|\ge \epsilon e^{-CT},\]
where the infimum is taken on $[0,1]^2\cap \{(t,s)\in Z_\epsilon: t>t_++\epsilon \ {\rm or} \ t<t_--\epsilon\}$.
 The constant  $C>0$ is independent of $\epsilon$ and $T$.
\end{lemma}
\begin{lemma}\label{lemma2}For every muti-index $\alpha=(\alpha_1,\alpha_2)$,
\[\|D^\alpha\phi\|_\infty\le e^{C_\alpha T},\]
where the norm is taken on the unit square $[0,1]^2$. The constant $C_\alpha$ is independent of $T$.
\end{lemma}
We postpone the proof of the lemmas and finish proving Theorem \ref{thm1}. We always use $C$ to denote various positive constants independent of $\epsilon$ and $T$. Recall that there are at most $O(e^{CT})$ summands with $\alpha\notin {\Gamma _{{{\rm T}_R}(\tilde \gamma )}}$. We claim that the kernel $K_\lambda^{osc}(t,s)$ of the operator $S_\lambda^{osc}$ is bounded by $e^{CT}(\epsilon\lambda+\epsilon^{-2}\lambda^{\frac34}+\epsilon^{-4}\lambda^{\frac12})$. Indeed, one can properly choose some smooth cutoff functions to decompose the domain $[0,1]^2$ and then apply Proposition \ref{oscint}, Lemma \ref{lemma0}-\ref{lemma2} and Young's inequality to the corresponding parts (1)-(4). Recall that Proposition \ref{oscint} consists of ``non-stationary phase'', ``left folds'' and ``right folds''. Since the estimate \eqref{amplitude} on the amplitude holds, it is not difficult to see that $\epsilon\lambda$ comes from Young's inequality, $\epsilon^{-2}\lambda^{\frac34}$ comes from one-side folds(or stationary phase), and  $\epsilon^{-4}\lambda^{\frac12}$ comes from non-stationary phase. Then Young's inequality gives
\begin{equation}\label{sosc}\|S_\lambda^{osc}\|_{L^2[0,1]\rightarrow L^{2}[0,1]}\le e^{CT}(\epsilon\lambda+\epsilon^{-2}\lambda^{\frac34}+\epsilon^{-4}\lambda^{\frac12}).\end{equation}Taking $T=c{\rm log \lambda}$ and $\epsilon=e^{-CT}T^{-1}$, where $c>0$ is a small constant ($c<(12C)^{-1}$), and combining \eqref{sosc} with the estimates on $S_\lambda^{tube}$ \eqref{stube} and $K_0$ \eqref{K_0}, we finish the proof.

\section{Proof of the Lemmas}
Before proving the lemmas, we remark that in the Poincar\'e half-space model
\[{{\rm T}_R}(\tilde\gamma)=\{(x,y,z)\in \mathbb{R}^3: z>0\ {\rm and}\ z\ge \sqrt{x^2+y^2}/\sqrt{({\rm cosh}R)^2-1}\}.\]
See Figure \ref{1}. Indeed, the distance between $(0,0,e^t)$ and $(x,y,z)$, $z>0$, is \[f(t)={\rm arcosh}\Big(1+\frac{x^2+y^2+(z-e^t)^2}{2ze^t}\Big)={\rm arcosh}\Big(\frac{x^2+y^2+z^2+e^{2t}}{2ze^t}\Big).\]
Setting $f'(t)=0$ gives $t={\rm ln}\sqrt{x^2+y^2+z^2}$, which must be the only minimum point. Thus the distance between $(x,y,z)$ and the infinite geodesic $\tilde\gamma$ is \[{\rm dist}((x,y,z),\tilde\gamma)={\rm arcosh}(\sqrt{1+(x/z)^2+(y/z)^2}).\]
Since ${\rm dist}((x,y,z),\tilde\gamma)\le R$ in ${{\rm T}_R}(\tilde\gamma)$, it follows that $z\ge \sqrt{x^2+y^2}/\sqrt{({\rm cosh}R)^2-1}$.
\begin{proof}[Proof of Lemma \ref{lemma1}]
First of all, we need to derive some useful results from the condition that $\phi(t,s)\le T$. Namely,
\begin{equation}\label{poly}(e^{2t}+d_1^2)e^{2s}-4r({\rm cosh}T)e^te^s+e^{2t}+d_2^2\le 0,\ t\in[0,1],\ s\in I.\end{equation}
Solving the quadratic inequality \eqref{poly} about $e^s$, we have \begin{equation}\label{segment2}\frac{r}{4{\rm cosh}T}\le e^s\le 4r{\rm cosh}T.\end{equation} The discriminant of \eqref{poly} has to be nonnegative:
\[16r^2({\rm cosh}T)^2e^{2t}-4(e^{2t}+d_1^2)(e^{2t}+d_2^2)\ge0,\]
from which we see that
\begin{equation}\label{a/rbound}\frac{a}{r}\le 2e{\rm cosh}T,\end{equation}
\begin{equation}\label{d1bound}d_1\le 2e{\rm cosh}T,\end{equation}
\begin{equation}\label{rlowbound}r\ge \frac1{2{\rm cosh}T},\end{equation}
which are similar to the observations in \cite[p.21]{xz}.

Moreover, to get the lower bounds of the derivatives, we need the condition that $\alpha  \notin {\Gamma _{{{\rm T}_R}(\tilde \gamma )}}$. We claim that there exists some constant $C$ independent of $T$ such that
\begin{equation}\label{claim32} \alpha  \notin {\Gamma _{{{\rm T}_R}(\tilde \gamma )}} \Rightarrow r\le C{\rm cosh}T\ {\rm or}\ d_1\ge \frac1{C{\rm cosh}T}.\end{equation}
Indeed, we are going to prove the contrapositive:
\begin{equation}\label{contra}  r\ge C{\rm cosh}T\ {\rm and}\ d_1\le \frac1{C{\rm cosh}T }\Rightarrow \alpha  \in {\Gamma _{{{\rm T}_R}(\tilde \gamma )}}.\end{equation}
We obtain this by showing that under the above assumptions on $r$ and $d_1$,  the segment $\gamma_2(s), s\in [-{\rm ln}(4r^{-1}{\rm cosh}T),{\rm ln}(4r{\rm cosh}T)]$ is completely included in ${{\rm T}_R}(\tilde \gamma )$, which implies $\alpha  \in {\Gamma _{{{\rm T}_R}(\tilde \gamma )}}$ by \eqref{segment2}. The argument is generalized from \cite[p.23]{xz}.
Solving the polynomial system
\[\begin{cases} z= \sqrt{x^2+y^2}/\sqrt{({\rm cosh}R)^2-1} \\(x,y,z)=(a+\frac{1-e^{2s}}{1+e^{2s}}r{\rm cos}\beta, \frac{1-e^{2s}}{1+e^{2s}}r{\rm sin}\beta,\frac{2re^s}{1+e^{2s}})\end{cases}\]
we can see that
\begin{equation}\label{intube}\begin{aligned}&\{\gamma_2(s):s\in \mathbb{R}\} \cap {{\rm T}_R}(\tilde \gamma )\\
&=\{\gamma_2(s):d_1^2e^{4s}+2(a^2+r^2-2({\rm cosh}R)^2r^2)e^{2s}+d_2^2\le 0\}.\end{aligned}\end{equation}
Note that \[\begin{cases} r\ge C {\rm cosh}T\\ d_1\le (C {\rm cosh}T)^{-1}\end{cases}\Rightarrow a/r\le 1+(C{\rm cosh}T)^{-2}\le \sqrt{({\rm cosh}R)^2-1}.\]
This implies
\[\frac{a}{r}\le \sqrt{\frac{({\rm cosh}R)^2-1}{({\rm cosh}R)^2-{\rm cos}^2\beta}}{\rm cosh}R,\]
which is equivalent to \[(a^2+r^2-2({\rm cosh}R)^2r^2)^2-d_1^2d_2^2\ge0.\]
This means that the discriminant of the  quadratic polynomial in terms of $e^{2s}$ in \eqref{intube} is nonnegative.
Thus when $d_1>0$, the RHS of (\ref{intube}) becomes
\begin{equation}\label{condit1}\{\gamma_2(s):u_- \le e^{2s} \le u_+\},\end{equation}
where \begin{equation}\label{condit2}u_\pm=\frac{2({\rm cosh}R)^2r^2-r^2-a^2\pm \sqrt{(a^2+r^2-2({\rm cosh}R)^2r^2)^2-d_1^2d_2^2}}{d_1^2}.\end{equation}
It is easy to see that
\begin{equation}\label{u-}u_-\le \frac{d_2^2}{2({\rm cosh}R)^2r^2-r^2-a^2}\le \frac{d_2^2}{({\rm cosh}R)^2r^2}\le \frac{({\rm cosh}R)^2+2{\rm cosh}R}{({\rm cosh}R)^2},\end{equation}
\begin{equation}\label{u+}u_+\ge \frac{(2({\rm cosh}R)^2-1)r^2-a^2}{d_1^2}\ge \frac{({\rm cosh}R)^2r^2}{d_1^2}.\end{equation}
So if we choose $C=4\sqrt{{\rm cosh}R+2}/\sqrt{{\rm cosh}R}$, we see that
 \begin{equation}\label{condit3}d_1>0 \ {\rm and}\  \begin{cases} r\ge C {\rm cosh}T\\ d_1\le (C {\rm cosh}T)^{-1}\end{cases} \Rightarrow \begin{cases}u_- \le r^2(4{\rm cosh}T)^{-2}\\ u_+\ge (4r{\rm cosh}T)^2\end{cases} \Rightarrow \alpha  \in {\Gamma _{{{\rm T}_R}(\tilde \gamma )}}.\end{equation}
In the easier case $d_1=0$, we have $u_+=+\infty$. Consequently, we obtain \eqref{contra}, which is equivalent to our claim \eqref{claim32}.

Moreover, we notice that by $\phi\le T$,\begin{equation}\label{phitimes}|\phi_{st}^{''}|\ge |\phi_{st}^{''}|\Big( \frac{A}{4re^{s+t}{\rm cosh}T}\Big)^2\ge \frac{|(a{\rm cos}\beta-r)(e^{2s+2t}+d_2^2)+(a{\rm cos}\beta+r)(e^{2t}+d_1^2e^{2s})|}{({\rm cosh}T)^2rA}.\end{equation}
Now we need to consider two cases: (I) $r\le a{\rm cos}\beta$; (II) $r>a{\rm cos}\beta$.

\noindent\textbf{Case (I)}: $\phi_{st}^{''}$ has no zeros and it is not difficult to obtain the lower bound of $|\phi_{st}^{''}|$. Indeed, if $d_1\ge 1$, by \eqref{phitimes} and \eqref{segment2}-\eqref{a/rbound}, we get
\[|\phi_{st}^{''}|\ge \frac{C(a{\rm cos}\beta+r)d_1^2r^2({\rm cosh}T)^{-2}}{({\rm cosh}T)^2r(d_1^2r^2({\rm cosh}T)^2)}\ge Ce^{-6T}.\]
If $d_1\le 1$, the claim \eqref{claim32} is needed. We assume that $r\le C{\rm cosh}T$. Then by \eqref{phitimes} and \eqref{segment2}-\eqref{rlowbound}, we obtain
\[|\phi_{st}^{''}|\ge \frac{C(a{\rm cos}\beta+r)e^{2t}}{({\rm cosh}T)^2r(r^2({\rm cosh}T)^2)}\ge Ce^{-6T}.\]
Otherwise, we assume that $d_1\ge (C{\rm cosh}T)^{-1}$. Then similarly we have
\[|\phi_{st}^{''}|\ge \frac{C(a{\rm cos}\beta+r)d_1^2r^2({\rm cosh}T)^{-2}}{({\rm cosh}T)^2r(r^2({\rm cosh}T)^2)}\ge Ce^{-8T}.\]
\noindent\textbf{Case (II)}: Since $\phi_{st}^{''}$ has zeros, we prove the lower bound of $|\phi_{st}^{''}|$ on $([0,1]\times I) \setminus Z_\epsilon$ first. The claim \eqref{claim32} is essential here. However, for technical reasons we only need a slightly weaker but useful version of the claim:
\begin{equation}\label{claim} \alpha  \notin {\Gamma _{{{\rm T}_R}(\tilde \gamma )}} \Rightarrow r\le C({\rm cosh}T)^7\ {\rm or}\ \begin{cases} d_1\ge (C{\rm cosh}T)^{-1}\\ r\ge C({\rm cosh}T)^7\end{cases}.\end{equation}
(i) Assume that $r\le C({\rm cosh}T)^7$.

 In this case, we use a ``restriction trick'' to reduce it to a one-variable problem. Let $\delta\in{\mathbb R}$. We restrict $\phi_{st}^{''}(t,s)$ on the straight line $s-t=\delta$ and obtain a uniform lower bound independent of $\delta$. Indeed,
\[\begin{aligned}&|(a{\rm cos}\beta-r)(e^{2s+2t}+d_2^2)+(a{\rm cos}\beta+r)(e^{2t}+d_1^2e^{2s})|\\
&=(r-a{\rm cos}\beta)|e^{2s+2t}-Y_0e^{2t}-X_0e^{2s}+d_2^2|\\
&=(r-a{\rm cos}\beta)|e^{4t}-(X_0+Y_0e^{-2\delta})e^{2t}+d_2^2e^{-2\delta}|e^{2\delta}\\
&=(r-a{\rm cos}\beta)|(e^{2t}-e^{2\tau_-})(e^{2t}-e^{2\tau_+})|e^{2\delta},
\end{aligned}\]
where \[2e^{2\tau_\pm}=X_0+Y_0e^{-2\delta}\pm\sqrt{(X_0-Y_0e^{-2\delta})^2+4Be^{-2\delta}}.\]
If $r-a{\rm cos}\beta\le \frac{r+a{\rm cos}\beta}{100}e^{-2\delta}$, then
\[2e^{2\tau_+}\ge Y_0e^{-2\delta}\ge 100.\]
But $t\in[0,1]$ implies that \[|e^{2t}-e^{2\tau_+}|\ge \frac12e^{2\tau_+}\ge \frac14Y_0e^{-2\delta}.\]
Let $Z_{\epsilon,\delta}= \{t\in \mathbb{R}: (t,t+\delta)\in Z_\epsilon\}$. Since the straight line $s-t=\delta$ is parallel to the ``major axis'' of $Z$,  we have
\begin{equation}\label{eps1}{\rm dist}(\tau_\pm, [0,1]\setminus Z_{\epsilon,\delta})\ge \epsilon/\sqrt2.\end{equation}
See Figure \ref{dist1}.
\begin{figure}
  \centering
    \includegraphics[width=15cm]{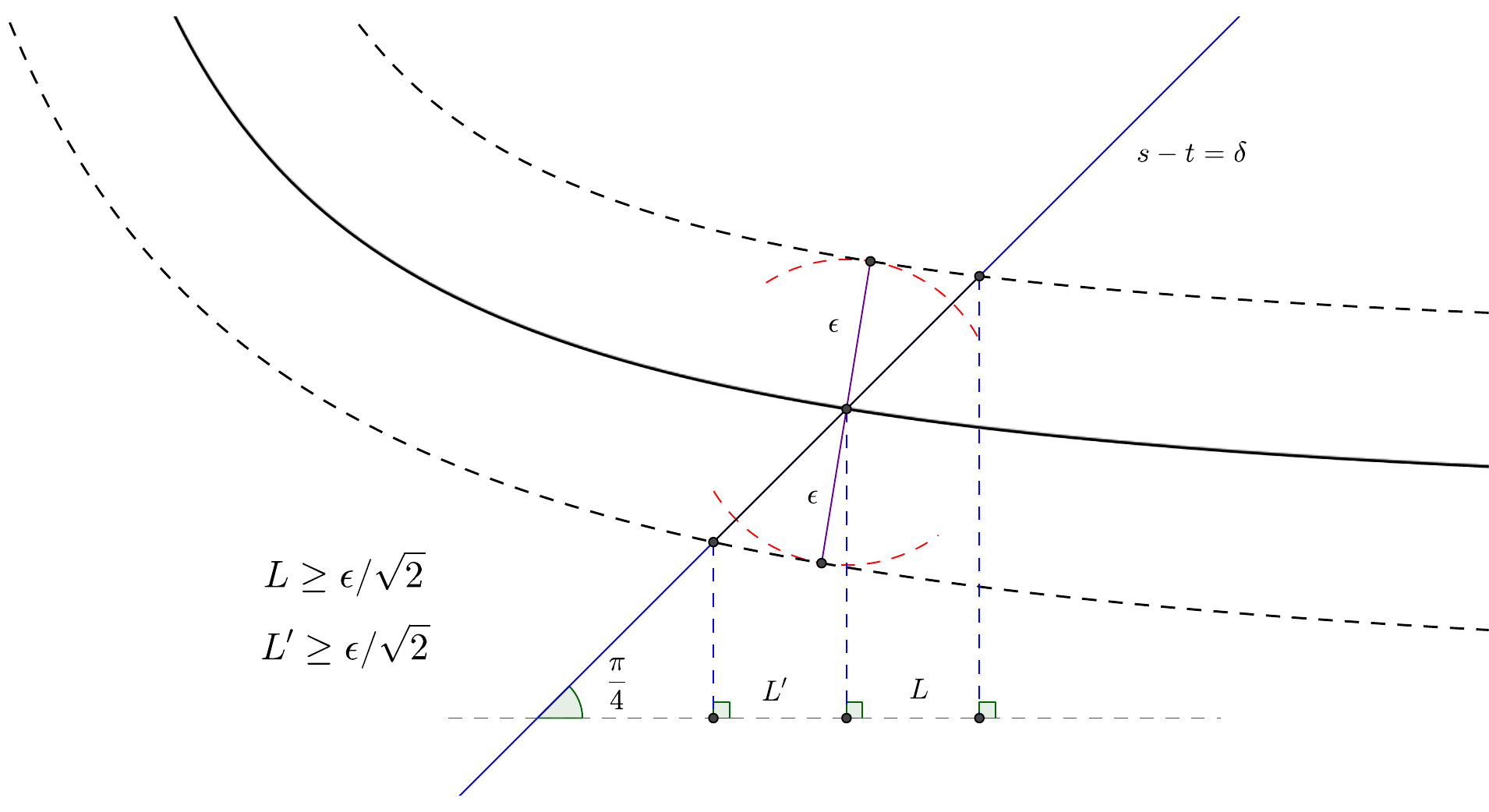}
\caption{Restriction on $s-t=\delta$}
  \label{dist1}
\end{figure}
This implies \[|e^{2t}-e^{2\tau_-}|\ge 1-e^{-\epsilon\sqrt2}\ge \epsilon/10,\  {\rm for}\ t\in [0,1]\setminus Z_{\epsilon,\delta}.\] Thus
\[(r-a{\rm cos}\beta)|(e^{2t}-e^{2\tau_-})(e^{2t}-e^{2\tau_+})|e^{2\delta}\ge \frac{\epsilon}{40}(r+a{\rm cos}\beta).\]
If $r-a{\rm cos}\beta\ge \frac{r+a{\rm cos}\beta}{100}e^{-2\delta}$, then we use \eqref{eps1} again to see that
 \[|e^{2t}-e^{2\tau_\pm}|\ge 1-e^{-\epsilon\sqrt2}\ge \epsilon/10,\  {\rm for}\ t\in [0,1]\setminus Z_{\epsilon,\delta},\]
 which gives \[(r-a{\rm cos}\beta)|(e^{2t}-e^{2\tau_-})(e^{2t}-e^{2\tau_+})|e^{2\delta}\ge \frac{\epsilon^2}{10000}(r+a{\rm cos}\beta).\] So we can use \eqref{phitimes}, \eqref{segment2}-\eqref{rlowbound} and our assumption $r\le C({\rm cosh}T)^7$ to obtain the lower bound of $|\phi_{st}^{''}|$, namely
\begin{equation}\label{lowbd1} |\phi_{st}^{''}|\ge \frac{C\epsilon^2(r+a{\rm cos}\beta)}{({\rm cosh}T)^2r(r^2({\rm cosh}T)^4)}\ge C\epsilon^2e^{-20T}.\end{equation}
(ii) Assume that $d_1\ge (C{\rm cosh}T)^{-1}\ {\rm and}\ r\ge C({\rm cosh}T)^7$.

If $|r-a{\rm cos}\beta|\le 1$, we can use \eqref{segment2}-\eqref{rlowbound} and our assumption to get
\[|(a{\rm cos}\beta-r)(e^{2s+2t}+d_2^2)+(a{\rm cos}\beta+r)(e^{2t}+d_1^2e^{2s})|\ge Cr^3({\rm cosh}T)^{-4},\]
since $(r+a{\rm cos}\beta)(d_1^2e^{2s}+e^{2t})\ge Cr^3({\rm cosh}T)^{-4}$ and $(r-a{\rm cos}\beta)(e^{2s+2t}+d_2^2)\le Cr^2({\rm cosh}T)^2$.

If $|r-a{\rm cos}\beta|\ge 1$, then $d_1\ge |r-a{\cos}\beta|\ge1$. Thus, $(r+a{\rm cos}\beta)(d_1^2e^{2s}+e^{2t})\ge Cr^3({\rm cosh}T)^{-2}$ and $(r-a{\rm cos}\beta)(e^{2s+2t}+d_2^2)\le Cr^2({\rm cosh}T)^3$, which imply
\[|(a{\rm cos}\beta-r)(e^{2s+2t}+d_2^2)+(a{\rm cos}\beta+r)(e^{2t}+d_1^2e^{2s})|\ge Cr^3({\rm cosh}T)^{-2}.\]
Therefore, we use \eqref{phitimes} and \eqref{segment2}-\eqref{rlowbound} to get
\begin{equation}\label{lowbd} |\phi_{st}^{''}|\ge \frac{Cr^3({\rm cosh}T)^{-4}}{({\rm cosh}T)^2r(r^2({\rm cosh}T)^4)}\ge Ce^{-10T},\end{equation}
which is better than the bound $\epsilon^2e^{-CT}$. Since the lower bounds in \eqref{lowbd1} and \eqref{lowbd} are independent of $\delta$, we finish the proof of the lower bound of $|\phi_{st}^{''}|$ on $([0,1]\times I) \setminus Z_\epsilon$.

Now we are ready to give the proof of the lower bounds of $|\phi_{st}^{''}/(t-t_c)|$ and $|\phi_{st}^{''}/(s-s_c)|$. Denote \[\epsilon_0=\frac12{\rm ln}\Big(1+\sqrt{\frac{B}{X_0Y_0}}\Big)+\epsilon.\]
\noindent\textbf{Part 1}: Assume that \begin{equation}\label{assume1}([0,1]\times I)\cap \{(t,s)\in Z_\epsilon: s>s_++\epsilon \ {\rm or} \ s<s_--\epsilon\}\ne \emptyset. \end{equation}We need to obtain the lower bound of $|\phi_{st}^{''}/(t-t_c)|$ on this set.
A simple computation using \eqref{defXYB}-\eqref{tspn} shows that  \begin{equation}\begin{aligned}s>s_++\epsilon\  &\Leftrightarrow\  e^{2s}>Y_0e^{2\epsilon_0},\\
s<s_--\epsilon\ &\Leftrightarrow\ e^{2s}<(Y_0-B/X_0)e^{-2\epsilon_0}.\end{aligned}\end{equation}
Hence \[|e^{2s}-Y_0|\ge (1-e^{-2\epsilon_0})Y_0.\]
Since the ``major axis'' of $Z$ is parallel to the straight line $s-t=0$, by our assumption \eqref{assume1} we have $t_c\in[-\epsilon\sqrt2,1+\epsilon\sqrt2]$. See Figure \ref{dist2}.
\begin{figure}
  \centering
    \includegraphics[width=15cm]{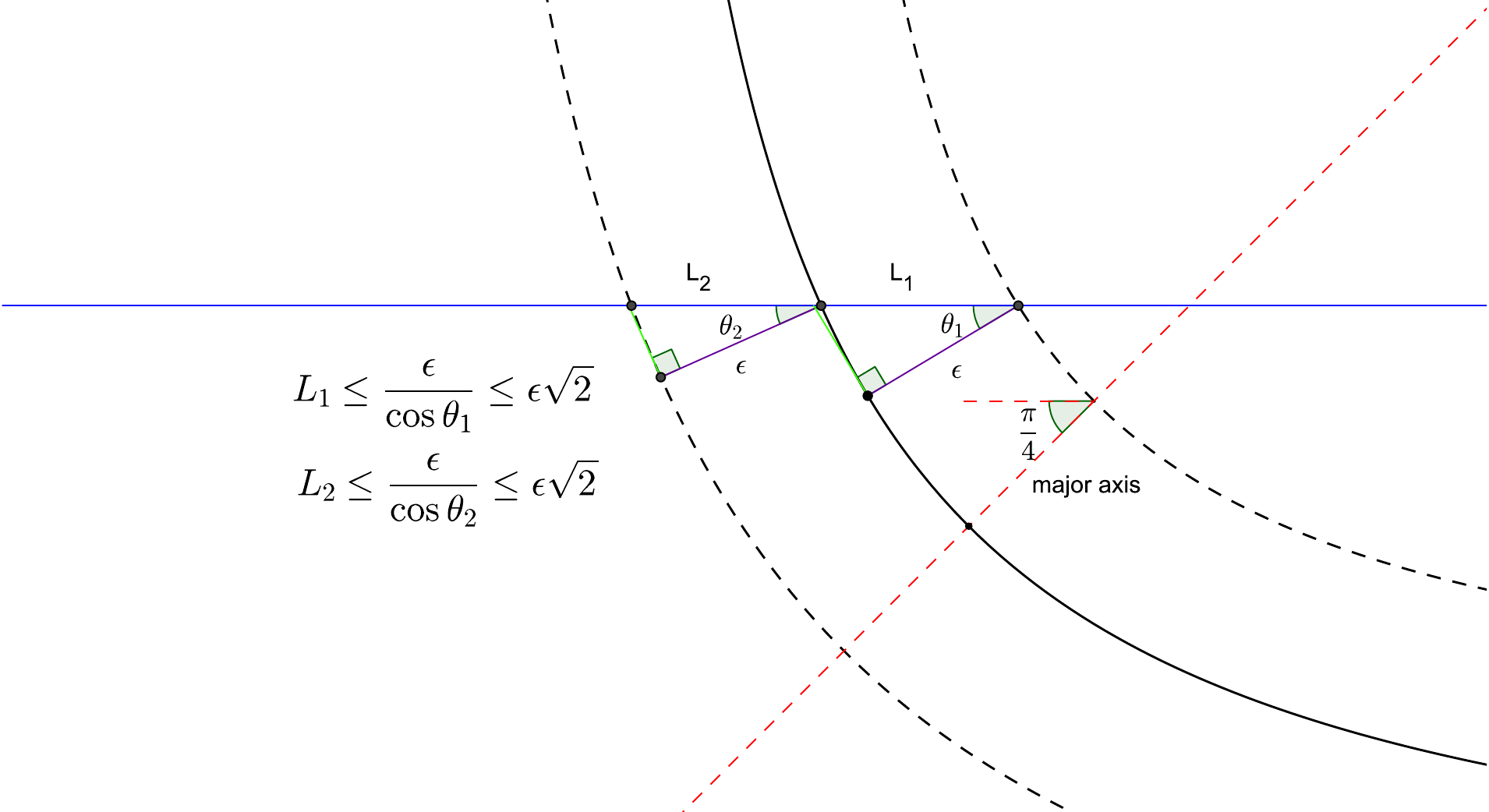}
\caption{${\rm dist}(t_c,[0,1])\le \epsilon\sqrt2$}
  \label{dist2}
  \centering
    \includegraphics[width=15cm]{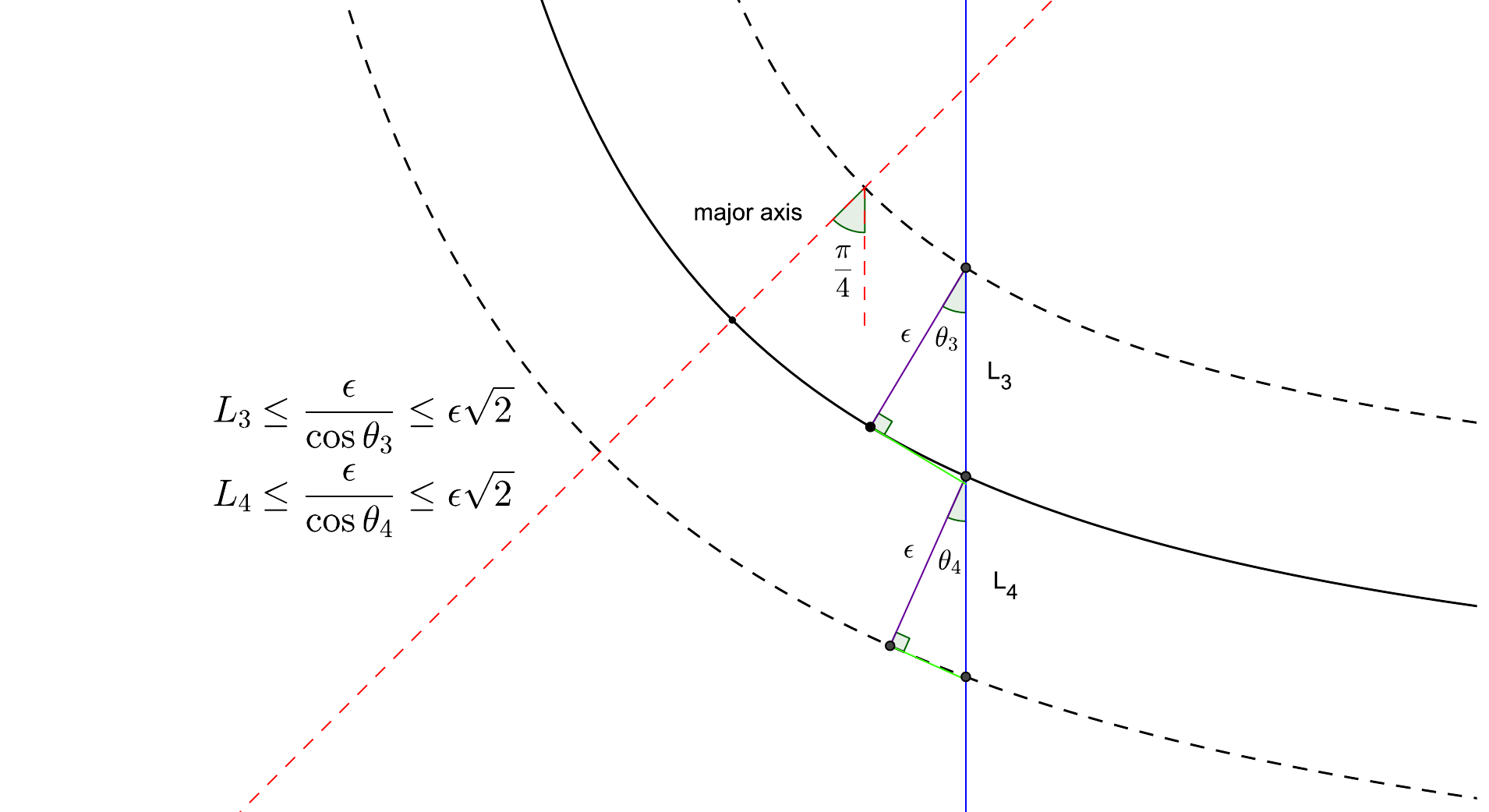}
\caption{${\rm dist}(s_c,I)\le \epsilon\sqrt2$}
  \label{dist3}
\end{figure}Thus,
\[\begin{aligned}&|(a{\rm cos}\beta-r)(e^{2s+2t}+d_2^2)+(a{\rm cos}\beta+r)(e^{2t}+d_1^2e^{2s})|/|t-t_c|\\
&=(r-a{\rm cos}\beta)\Big|\frac{(e^{2t}-X_0)(e^{2s}-Y_0)-B}{t-t_c}\Big|\\
&=(r-a{\rm cos}\beta)\Big|\frac{(e^{2t}-e^{2t_c})(e^{2s}-Y_0)}{t-t_c}\Big|\\
&=(r-a{\rm cos}\beta)\cdot2e^{2t'}\cdot|e^{2s}-Y_0|\\
&\ge (r-a{\rm cos}\beta)\cdot2e^{-2\epsilon\sqrt2}\cdot(1-e^{-2\epsilon_0})Y_0\\
&\ge \frac{\epsilon}{100}(r+a{\rm cos}\beta),
\end{aligned}\]
where we use the mean value theorem and $\epsilon_0\ge \epsilon$.

First, we assume that $r\le C({\rm cosh}T)^7$. Then using \eqref{phitimes} and \eqref{a/rbound}-\eqref{rlowbound}, we obtain
\[\Big|\frac{\phi_{st}^{''}}{t-t_c}\Big|\ge \frac{C\epsilon(r+a{\rm cos}\beta)}{({\rm cosh}T)^2r(r^2({\rm cosh}T)^4)}\ge C\epsilon e^{-20T}. \]
Under the other assumption that ``$d_1\ge (C{\rm cosh}T)^{-1}\ {\rm and}\ r\ge C({\rm cosh}T)^7$'', since $|t-t_c|\le 1+\epsilon\sqrt2$ and the lower bound \eqref{lowbd} of $|\phi_{st}^{''}|$ is still applicable here, we get
\[\Big|\frac{\phi_{st}^{''}}{t-t_c}\Big|\ge Ce^{-10T}.\]
\noindent\textbf{Part 2}:  Assume that \begin{equation}\label{assume2}([0,1]\times I)\cap \{(t,s)\in Z_\epsilon: t>t_++\epsilon \ {\rm or} \ t<t_--\epsilon\}\ne \emptyset.\end{equation} We need to get the lower bound of $|\phi_{st}^{''}/(s-s_c)|$ on this set.
 It is also not difficult to see from \eqref{defXYB}-\eqref{tspn} that \begin{equation}\begin{aligned}t>t_++\epsilon\ &\Leftrightarrow\ e^{2t}> X_0e^{2\epsilon_0},\\
 t<t_--\epsilon\ &\Leftrightarrow\ e^{2t}<(X_0-B/Y_0)e^{-2\epsilon_0}.\end{aligned}\end{equation}
  Hence
 \[|e^{2t}-X_0|\ge (1-e^{-2\epsilon_0}){\rm max}\{X_0,1\}.\]
If $t>t_++\epsilon$, clearly we have $e^{2s_c}\ge Y_0$. See Figure \ref{3}. If $B=0$, we have $e^{2s_c}= Y_0$.  If $t<t_--\epsilon$ and $B>0$, then from \eqref{tcsc} we get
\[\begin{aligned}&e^{2s_c}= Y_0-\frac{B}{X_0-e^{2t}}\ge Y_0-\frac{B}{e^{2t+2\epsilon_0}+B/Y_0-e^{2t}}\\
&= Y_0-\frac{B}{e^{2t}(e^{2\epsilon}-1)+e^{2t+2\epsilon}\sqrt{B/(X_0Y_0)}+B/Y_0}\ge Y_0-\frac{B}{\sqrt{B/(X_0Y_0)}+B/Y_0}\\
&= Y_0-\frac{X_0Y_0}{\sqrt{X_0Y_0/B}+X_0}\ge Y_0-\frac{X_0Y_0}{1+X_0}\\
&= \frac{Y_0}{X_0+1},
\end{aligned}\]
where we use $X_0Y_0/B>1$ from \eqref{XYB}. Since the ``major axis'' of $Z$ is parallel to the straight line $s-t=0$, by the assumption \eqref{assume2} we get ${\rm dist}(s_c,I)\le \epsilon\sqrt2$. See Figure \ref{dist3}. Therefore,
\[\begin{aligned}&|(a{\rm cos}\beta-r)(e^{2s+2t}+d_2^2)+(a{\rm cos}\beta+r)(e^{2t}+d_1^2e^{2s})|/|s-s_c|\\
&=(r-a{\rm cos}\beta)\Big|\frac{(e^{2t}-X_0)(e^{2s}-Y_0)-B}{s-s_c}\Big|\\
&=(r-a{\rm cos}\beta)\Big|\frac{(e^{2t}-X_0)(e^{2s}-e^{2s_c})}{s-s_c}\Big|\\
&=(r-a{\rm cos}\beta)|e^{2t}-X_0|\cdot2e^{2s'}\\
&\ge(r-a{\rm cos}\beta)|e^{2t}-X_0|\cdot2e^{2(s_c-1-\epsilon\sqrt2)}\\
&\ge(r-a{\rm cos}\beta)(1-e^{-2\epsilon_0}){\rm max}\{X_0,1\}\cdot\frac{2Y_0}{X_0+1}e^{-2-2\epsilon\sqrt2}\\
&\ge \frac{\epsilon}{100}(r+a{\rm cos}\beta),
\end{aligned}\]
where we use the mean value theorem and $\epsilon_0\ge \epsilon$. Then we can obtain the lower bound of $|\phi_{st}^{''}/(s-s_c)|$ in the same way as Part 1. First, under the assumption that $r\le C({\rm cosh}T)^7$, we have
\[\Big|\frac{\phi_{st}^{''}}{s-s_c}\Big|\ge \frac{C\epsilon(r+a{\rm cos}\beta)}{({\rm cosh}T)^2r(r^2({\rm cosh}T)^4)}\ge C\epsilon e^{-20T}. \]
Under the other assumption that ``$d_1\ge (C{\rm cosh}T)^{-1}\ {\rm and}\ r\ge C({\rm cosh}T)^7$'', noting that $|s-s_c|\le 1+\epsilon\sqrt2$ and the bound \eqref{lowbd} is still valid here, we get
\[\Big|\frac{\phi_{st}^{''}}{s-s_c}\Big|\ge Ce^{-10T}.\] So far we have finished the proof of all the lower bounds.
\end{proof}
\begin{proof}[Proof of Lemma \ref{lemma2}]
We only need to prove the upper bounds of mixed derivatives when $\alpha\ne Identity$, since the bounds for pure derivatives are well known in \cite{berard}, \cite{top} and we do not use them in this paper. For convenience, we denote
\[G(t,s)=(a{\rm cos}\beta-r)(e^{2s+2t}+d_2^2)+(a{\rm cos}\beta+r)(e^{2t}+d_1^2e^{2s}),\]
\[E(t,s)=A^2-16r^2e^{2s+2t}.\]Recalling the formula \eqref{phist}, we have $\phi_{st}^{''}=16re^{2s+2t}GE^{-3/2}$.
By induction it is not difficult to see that for any muti-index $\alpha=(\alpha_1,\alpha_2)$
\[D^\alpha\Big(\frac{G}{E^\gamma}\Big)=E^{-\gamma-|\alpha|}\sum\limits_{0\le |\beta_0|+\cdot\cdot\cdot+|\beta_{|\alpha|}|\le |\alpha|}{C_{\gamma,\alpha,\beta_0,...,\beta_{|\alpha|}}D^{\beta_0}G\cdot D^{\beta_1}E\cdot\cdot\cdot D^{\beta_{|\alpha|}}E},\]
where $|\alpha|=\alpha_1+\alpha_2$, and $C_{\gamma,\alpha,\beta_0,...,\beta_{|\alpha|}}$ are constants independent of $G$ and $E$.
Thus,
\[D^\alpha\phi_{st}^{''}=\frac{re^{2s+2t}}{E^{3/2+|\alpha|}}\sum\limits_{0\le |\beta_0|+\cdot\cdot\cdot+|\beta_{|\alpha|}|\le |\alpha|}{C_{\alpha,\beta_0,...,\beta_{|\alpha|}}D^{\beta_0}G\cdot D^{\beta_1}E\cdot\cdot\cdot D^{\beta_{|\alpha|}}E}.\]
From the condition that $\phi(t,s)\ge 2$, we have $A\ge 4({\rm cosh}2)re^{s+t}$. Thus,
\[A-4re^{s+t}\ge (4{\rm cosh}2-4)re^{s+t}.\]
If $r\ge C{\rm cosh}T$, then by \eqref{segment2}-\eqref{rlowbound}, \[E\ge (A-4re^{s+t})^2\ge Cr^2e^{2s+2t}\ge Cr^4({\rm cosh}T)^{-2}, \] \[|D^\alpha E|\le C_\alpha r^4({\rm cosh}T)^8,\quad |D^\alpha G|\le C_\alpha r^3({\rm cosh}T)^5.\]
Hence,
\[|D^\alpha\phi_{st}^{''}|\le \frac{C_\alpha r(r{\rm cosh}T)^2}{(r^4({\rm cosh}T)^{-2})^{3/2+|\alpha|}}r^3({\rm cosh}T)^5(r^4({\rm cosh}T)^8)^{|\alpha|}\le C_\alpha e^{(10|\alpha|+10)T},\]
If $r\le C{\rm cosh}T$, then by \eqref{segment2}-\eqref{rlowbound}, \[E\ge (A-4re^{s+t})^2\ge Cr^2e^{2s+2t}\ge  C({\rm cosh}T)^{-6}, \] \[|D^\alpha E|\le C_\alpha ({\rm cosh}T)^{12},\quad |D^\alpha G|\le C_\alpha ({\rm cosh}T)^8.\]
Therefore,
\[|D^\alpha\phi_{st}^{''}|\le \frac{C_\alpha ({\rm cosh}T)^5}{(({\rm cosh}T)^{-6})^{3/2+|\alpha|}}({\rm cosh}T)^8(({\rm cosh}T)^{12})^{|\alpha|}\le C_\alpha e^{(18|\alpha|+22)T}.\]\end{proof}
\begin{remark}The condition that $\alpha  \notin {\Gamma _{{{\rm T}_R}(\tilde \gamma )}}$ is essential in the proof of the lower bounds. However, the proof of the upper bounds only needs $2\le \phi\le T$.
\end{remark}
\section{Acknowledgement}
We would like to thank Professor C. Sogge for his guidance and patient discussions during this study. We are grateful to
him for pointing us to the role of the more favorable dispersion for the wave equation on hyperbolic space, which inspires us to get further improvements on an earlier version of this paper. It's our pleasure to thank Professor M. Blair for his helpful comments on the preprint and our colleagues X. Wang and Y. Xi for fruitful discussions. We also wish to  thank  the anonymous referee for valuable comments and suggestions.
\bibliography{3dER}

\bibliographystyle{plain}

\end{document}